\documentclass [11pt,a4paper]{amsart}
\usepackage{amsmath}
\usepackage{amsthm}
\usepackage{amssymb}
\usepackage{amscd}
\usepackage{amsfonts}
\usepackage{amsbsy}
\usepackage{epsfig,afterpage}
\usepackage[dvips]{psfrag}
\usepackage{graphicx}
\usepackage{indentfirst, latexsym, bm,amssymb}
\usepackage{bbding}
\usepackage{epsfig}
\usepackage{enumitem}
   \newcommand{\lb}{label}
   \newcommand{\lm}{leftmargin}
   
   \newcommand{\is}{itemsep}
   
\usepackage{xcolor}
\definecolor{blue1}{RGB}{0,36,107}

\newtheorem {theorem} {Theorem}

\newtheorem {lemma}  [theorem]{Lemma}
\newtheorem {example} [theorem]{Example}
\newtheorem {remark} [theorem]{Remark}

\newcommand{\X}{{\mathcal{\bf} X}}
\newcommand{\e}{{\text e}}
\newcommand{\N}{{\mathbb N}}

\newcommand{\R}{{\mathbb R}}

\def\eps{\varepsilon}

\newcommand{\SigmaND}{\Sigma^{\scriptscriptstyle \text{ND}}}
\newcommand{\rhoND}{\rho^{\scriptscriptstyle \text{ND}}}

\newcommand{\SigmaD}{\Sigma^{\scriptscriptstyle \text{D}}}
\newcommand{\rhoD}{\rho^{\scriptscriptstyle \text{D}}}

\newcommand{\SigmaNDp}{\Sigma^{\scriptscriptstyle \text{ND+}}}

\newcommand{\SigmaDp}{\Sigma^{\scriptscriptstyle \text{D+}}}

\DeclareMathOperator\sign{sgn}
\DeclareMathOperator{\imagem}{Im}
\DeclareMathOperator{\rank}{rank}
\DeclareMathOperator{\sg}{span}

\title{Nonuniform $\mu$-dichotomy spectrum and kinematic similarity}

\author{César M. Silva}
\email{csilva@ubi.pt}
\thanks{C\'esar M. Silva was partially supported by FCT through CMA-UBI (project UIDB/00212/2020). The author wishes to thank the anonymous referee for the very helpful questions raised and comments made.}

\begin{document}

\maketitle

\begin{abstract}
For linear nonautonomous differential equations we introduce a new family of spectrums defined with general nonuniform dichotomies: for a given growth rate $\mu$ in a large family of growth rates, we consider a notion of spectrum, named nonuniform $\mu$-dichotomy spectrum. This family of spectrums contain the nonuniform dichotomy spectrum as the very particular case of exponential growth rates. For each growth rate $\mu$, we describe all possible forms of the nonuniform $\mu$-dichotomy spectrum, relate its connected components with adapted notions of Lyapunov exponents, and use it to obtain a reducibility result for nonautonomous linear differential equations. We also give an illustrative examples where the spectrum is obtained, including a situation where a normal form is obtained for polynomial behavior.
\end{abstract}

\section{Introduction}\label{section:introduction}

The dichotomy spectrum (also called dynamical spectrum and Sacker-Sell spectrum), defined with exponential dichotomies, was introduced by Sacker and Sell and used to study linear skew product flows with compact base~\cite{Sacker-Sell-JDE-1978}. The usefulness of the dichotomy spectrum has been widely proved in several contexts~\cite{Chicone-Latushkin-MSM-1999,Johnson-Palmer-Sell-SIAMJMA-1987,Colonius-Kliemann-DC-2000}. In particular, the dichotomy spectral theory has proved to be an important tool in the obtention of normal forms for nonautonomous differential equations~\cite{Siegmund-JLMS-2002,Siegmund-JDE-2002}. A version of dichotomy spectrum for nonautonomous linear difference equations was considered and a spectral theorem in that context was proved in~\cite{Aulbach-Siegmund-JDEA-2001}. An infinite-dimensional version of the dichotomy spectrum was considered by Sacker and Sell~\cite{Sacker-Sell-JDE-1994} and by Chow and Leiva~\cite{Chow-Leiva-JJIAM-1994,Chow-Leiva-JDE-1996,Chow-Leiva-JDE-1995}. We also refer the interesting papers~\cite{Potzche-JDEA-2009,Potzche-IEOP-2012} by P\"{o}tzsche, where the relation of the dichotomy spectrum with a weighted shift operator on some sequence space is explored. The book~\cite{Chicone-Latushkin-MSM-1999} by Chicone and Latushkin is a central reference in the theory of differential equations in Banach spaces via spectral properties of the associated evolution semigroup.

Despite its importance in the theory, the notion of (uniform) exponential dichotomy is sometimes too restrictive and it is important to consider more general hyperbolic behavior. The concept of nonuniform exponential dichotomy generalizes the concept of (uniform) exponential dichotomy by allowing some exponential loss of hyperbolicity along the trajectories. A very complete theory of nonuniform exponential dichotomies is being developed by Barreira and Valls and a vast amount of results have already been published by those authors concerning this subject. Still in the context of nonuniform exponential dichotomies, we refer the books~\cite{Barreira-Valls-livro1}, by Barreira and Valls, concerning stability theory, and \cite{Barreira-Dragicevich-Valls-livro}, by Barreira, Dragičević and Valls, concerning the relation of admissibility with hyperbolic behavior.

To the best of our knowledge, a version of dichotomy spectrum defined with nonuniform exponential dichotomies, the so called nonuniform dichotomy spectrum, was considered for the first time by Zhang in~\cite{Zhang-JFA-2014} and by Chu, Liao, Siegmund, Xia and Zhang in~\cite{Chu-Liao-Siegmund-Xia-Zhang-BSM-2015}, with a slightly different definition. In both papers the authors describe the topological structure of the nonuniform dichotomy spectrum and use this spectrum to prove a result on the kinematic similarity of nonautonomous linear equations and block diagonal systems. Additionally, in~\cite{Zhang-JFA-2014}, the author obtained normal forms for nonautonomous nonlinear systems using the spectrum of the linear part. Corresponding results in the context of discrete dynamics were obtained in a recent paper by Chu, Liao, Siegmund, Xia and Zhu~\cite{Chu-Liao-Siegmund-Xia-Zhu-ANA-2022}.

Another way to generalize the notion of exponential dichotomy is to assume that the asymptotic behavior is not exponential. This reasoning leads to uniform dichotomies with asymptotic behavior given by general growth rates, a notion considered in the work of Naulin and Pinto~\cite{Naulin-Pinto-JDE-1995,Pinto-CMA-1994}. Proceeding in the direction of more general behavior, we can consider dichotomies that are both nonuniform and do not necessarily have exponential growth.

In~\cite{Bento-Silva-QJM-2012,Bento-Silva-NA-2012} a very general notion of nonuniform dichotomy, the notion of nonuniform $(\mu,\nu)$-dichotomy, was considered and stable manifold theorems were established for perturbations of nonautonomous linear equations admiting this type of hyperbolic behavior, both in the continuous and discrete settings. In the present work we consider nonuniform $\mu$-dichotomies, a general notion of dichotomy obtained by assuming in the definition of $(\mu,\nu)$-dichotomy that $\mu=\nu$. Note that the notion of $\mu$-dichotomy includes the usual nonuniform exponential dichotomies in~\cite{Barreira-Valls-JDE-2006,Barreira-Valls-CMP-2005} as well as the nonuniform polynomial dichotomies, introduced independently in~\cite{Bento-Silva-JFA-2009} and \cite{Barreira-Valls-NA-2009}, as very particular cases.

There is already a considerable amount of papers devoted to the study of difference and differential equations under the hypothesis of existence of a generalized notion of nonuniform dichotomy: in~\cite{Chang-Zhang-Qin-JMAA-2012}, Chang, Zhang and Qin discussed the robustness of nonuniform $(\mu,\nu)$-dichotomies and, in~\cite{Chu-BSM-2013}, Chu adressed the same problem for the discrete-time case (see also Crai~\cite{Crai-ASM-2017} for related results); in~\cite{Pan-JM-2014}, Pan discussed the existence of stable manifolds for delay equations with nonuniform $(\mu,\nu)$-dichotomies; in~\cite{Bento-Lupa-Megan-Silva-DCDS-B-2017}, Bento, Lupa, Megan and Silva obtained integral conditions for the existence of a nonuniform $\mu$-dichotomy (see also~\cite{Boruga-Megan-Toth-A-2021} where Boruga, Megan and Toth obtained integral characterizations of a generalized notion of uniform stability, in the spirit of Barbashin and Datko); in~\cite{Dragicevic-Pecek-Lupa-EJQTDE-2020} Dragičević, Peček and Lupa introduced the notion of a $\mu$-dichotomy with respect to a family of norms (a notion that generalizes the notion of nonuniform $\mu$-dichotomy) and characterized it in terms of two admissibility conditions (see also~\cite{Silva-CPAA-2021} for related results in discrete time context); in~\cite{Bento-Vilarinho-JDDE-2021} Bento and Vilarinho proved the existence of measurable invariant manifolds for small perturbations of linear random dynamical systems admitting a general type of dichotomy; in~\cite{Bento-Costa-EJQTDE-2017} Bento and Costa established conditions for the existence of global Lipschitz invariant center manifolds for perturbations of linear nonautonomous equations admitting a very general form of nonuniform trichotomy; in~\cite{Backes-Dragicevic-2021} Backes and Dragičević obtained a shadowing result for perturbations of linear equations admiting a nonuniform $(\mu,\nu)$-dichotomy.

There are two main objectives in this paper. The first objective is to define and characterize a new notion a nonuniform spectrum, defined with nonuniform $\mu$-dichotomies. To the extent of our knowledge, a notion of dichotomy spectrum defined with dichotomies with nonexponential behavior (even in the uniform case) is considered for the first time in the present work (see the definitions of nonuniform $\mu$-dichotomy spectrum and of $\mu$-dichotomy spectrum in section~\ref{section:Nonunif-dichotomy-spectrum}). We note that the equations involved in the definition of the new spectrum depend on the growth rates of the dichotomy and therefore the new definition is not evident from the definition of nonuniform dichotomy spectrum. The second objective is related to normal form theory: we use our notion of spectrum to obtain a reducibility theorem for linear nonautonomous differential equations and after apply our results to obtain the normal form of a triangular system with polynomial behavior. We also relate our notions of spectrum to the notion of Lyapunov exponent adapted to a growth rate $\mu$, a notion already considered in the literature~\cite{Barreira-Valls-NA-2009, Barreira-Valls-DCDS-2012}. We emphasise that our notion of spectrum allows us to obtain normal forms in situations where the usual Lyapunov exponents are zero. This is illustrated in the triangular example referred.

The theory on normal form can be traced back to Poincaré~\cite{Poincare-JMPA-1928}. The aim of this theory is to find a suitable change of variables that transforms a system of ordinary differential equations in another system that is simpler to analyse and has the same qualitative behavior. In the context of nonautonomous systems, reducibility results were obtained in~\cite{Siegmund-JLMS-2002} using the dichotomy spectrum and in~\cite{Zhang-JFA-2014,Chu-Liao-Siegmund-Xia-Zhang-BSM-2015,Chu-Liao-Siegmund-Xia-Zhu-ANA-2022} using the nonuniform dichotomy spectrum. We also refer the series of papers~\cite{Barreira-Valls-JDE-2006-2,Barreira-Valls-JFA-2007,Barreira-Valls-JDDE-2007,Barreira-Valls-ETDS-2008}, by Barreira and Valls, where several results on the topological conjugacy between systems with nonuniformly exponential behaviour were obtained. Finally, we mention the interesting work of Li, Llibre and Wu~\cite{Li-Llibre-Wu-ETDS-2009} concerning normal forms of almost periodic differential systems and the related work~\cite{Li-Llibre-Wu-JDEA-2009}, by the same authors, devoted to the study of normal forms of almost periodic difference systems.

We now briefly describe the structure of the paper. In section~\ref{section:Nonunif-dichotomy-spectrum} we present our setting and define the new family of spectrums. In section~\ref{section:spectral-theorems} we describe the topological structure of the new spectrums, relate them with adapted notions of Lyapunov exponents and present explicit examples of linear nonautonomous differential equations for which the new spectrums can be computed. In section~\ref{section:Normal formas} we obtain a reducibility theorem on the existence of normal forms for linear nonautonomous differential equations, using the new nonuniform $\mu$-dichotomy spectrum. Finally, in section~\ref{section:half_line}, we rewrite our results in the half-line setting and present an example of a family of nonautonomous triangular systems for which we can compute the nonuniform polynomial dichotomy spectrum and use it to obtain a normal form, highlighting the existence of linear integral manifolds where polynomial contraction and expansion occur.
\section{Nonuniform $\mu$-dichotomy spectrum}\label{section:Nonunif-dichotomy-spectrum}

Let $M_n(\R)$ be the set of square matrix functions of $n$th
order defined in $\R$ and let $A(t)\in M_n(\R)$ for each $t \in \R$. In this paper we consider nonautonomous linear systems
\begin{equation}\label{eq:sist-linear}
x'=A(t)x,
\end{equation}
$t \in \R$. We assume that $t \mapsto A(t)$ is continuous. As a consequence, all solutions of system \eqref{eq:sist-linear} are defined on the whole $\R$. We denote by $\Phi(t,s)$, $t,s \in \R$, the corresponding evolution operator.

We say that a function $\mu:\R\to \R^+$ is a {\it growth rate} if it is strictly increasing, $\mu(0)=1$, $\displaystyle \lim_{t \to +\infty} \mu(t)=+\infty$ and $\displaystyle \lim_{t \to -\infty} \mu(t)=0$. If, additionally, $\mu$ is differentiable, we say that it is a differentiable growth rate.

Denote the sign of $a$ by $\sign(a)$.
We say that system \eqref{eq:sist-linear} admits a {\it nonuniform $\mu$-dichotomy} (N$\mu$D)
if there is a family of projections $P(t)\in M_n(\R)$, $t\in\R$, such that, for all $t,s\in\R$,
$$P(t)\Phi(t,s)=\Phi(t,s)P(s),$$
and there are constants $K\ge 1$, $\alpha<0$, $\beta>0$ and $\theta,\nu\ge 0$, with $\alpha+\theta<0$ and
$\beta-\nu>0$, such that
\begin{equation}\label{eq:dichotomy1}
\|\Phi(t,s)P(s)\| \le  K \left(\frac{\mu(t)}{\mu(s)}\right)^{\alpha}\mu(s)^{\sign(s)\theta} \mbox{ for } t\ge s,
\end{equation}
\begin{equation}\label{eq:dichotomy2}
\|\Phi(t,s)Q(s)\| \le  K\left(\frac{\mu(t)}{\mu(s)}\right)^{\beta}\mu(s)^{\sign(s)\nu} \mbox{ for } t\le s,
\end{equation}
where $Q(s)=I-P(s)$ is the complementary projection. When $\theta=\nu=0$ we say that system \eqref{eq:sist-linear}
admits a {\it uniform $\mu$-dichotomy} (or more simply a {\it $\mu$-dichotomy} ($\mu$D)).

If $\mu(t)=\e^t$ we obtain the usual notion of {nonuniform exponential dichotomy} (NED).

Given a strictly increasing function $\nu:\R_0^+\to[1,+\infty[$ such that $\nu(0)=1$ and $\displaystyle \lim_{t \to +\infty} \nu(t)=+\infty$,  we can define a growth rate $\mu:\R \to \R$ by
\[
\mu(t)=\nu(|t|)^{\sign(t)}=
\begin{cases}
\nu(t) \ & \ t\ge 0\\
\frac{1}{\nu(|t|)} & \ t<0
\end{cases}
\]
and obtain the corresponding nonuniform $\mu$-dichotomy notion. Notice that when $\nu(t)=\e^t$ we obtain the already mentioned notion of nonuniform exponential dichotomy and when $\nu(t)=t+1$ we obtain the growth rate $p:\R \to \R$ given by
\begin{equation}\label{eq:poly-dich-R}
p(t)=
\begin{cases}
t+1 \ & \ t\ge 0\\
\frac{1}{1-t} & \ t<0
\end{cases}
\end{equation}
and the corresponding nonuniform $p$-dichotomy that we call a {\it nonuniform polynomial dichotomy (NPD)}. Note also that if $\nu$ is differentiable then $\mu$ is differentiable.
In particular, the polynomial growth rate $p$ is differentiable.

\begin{remark}\label{remark:form-of-projections}
  Consider the family of projections $P(t)\in M_n(\R)$ in~\eqref{eq:dichotomy1}--\eqref{eq:dichotomy2} and set
  \[
  \widetilde P=
  \left[
  \begin{array}{cc}
  I & 0\\
  0 & 0
  \end{array}
  \right],
  \]
  for all $t \in \R$, where $I$ denotes the identity of order $\dim \imagem P(t)$. A similar argument to that in~ Lemma 2.2 in~\cite{Chu-Liao-Siegmund-Xia-Zhang-BSM-2015}, shows that a fundamental matrix of~\eqref{eq:sist-linear}, $X(t)$, can be chosen appropriately so that
  $$\|\Phi(t,s)P(s)\|=\|X(t)\widetilde P X(s)^{-1}\| \quad \text{ and } \quad \|\Phi(t,s)Q(s)\|=\|X(t)\widetilde Q X(s)^{-1}\|.$$
  Thus, inequalities~\eqref{eq:dichotomy1}--\eqref{eq:dichotomy2} can be written as
\begin{equation}\label{eq:dichotomy1-fund-matrix}
\|X(t)\widetilde P X(s)^{-1}\| \le  K \left(\frac{\mu(t)}{\mu(s)}\right)^{\alpha}\mu(s)^{\sign(s)\theta} \ \mbox{ for } \ t\ge s,
\end{equation}
\begin{equation}\label{eq:dichotomy2-fund-matrix}
\|X(t)\widetilde Q X(s)^{-1}\| \le  K\left(\frac{\mu(t)}{\mu(s)}\right)^{\beta}\mu(s)^{\sign(s)\nu} \ \mbox{ for } \ t\le s,
\end{equation}
where $\widetilde Q=I-\widetilde P$ is the complementary projection of $\widetilde P$.

For the sake of completeness, we reproduce the argument in our context. Fix $\tau \in \R$. Then there is a non-singular matrix $T$ such that
$TP(\tau)T^{-1}=\widetilde P$ with
\[
\widetilde P  =
  \left[
  \begin{array}{cc}
  I & 0\\
  0 & 0
  \end{array}
  \right],
\]
where $I$ denotes the identity of order $\dim \imagem P(t)$. For each $t \in \R$, let $X(t)=\Phi(t,\tau)T^{-1}$. We have
\begin{equation}\label{eq:phi-t-s-P-s}
\begin{split}
\Phi(t,s)P(s)
&=\Phi(t,\tau)\Phi(\tau,s)P(s)=\Phi(t,\tau)P(\tau)\Phi(\tau,s)\\
&=\Phi(t,\tau)T^{-1}\widetilde P T\Phi(\tau,s)=\Phi(t,\tau)T^{-1}\widetilde P (\Phi(s,\tau)T^{-1})^{-1}\\
&=X(t)\widetilde P X(s)^{-1}
\end{split}
\end{equation}
and thus
\begin{equation}\label{eq:phi-t-s-Q-s}
\begin{split}
\Phi(t,s)Q(s)
& =\Phi(t,s)I-\Phi(t,s)P(s)=\Phi(t,\tau)\Phi(\tau,s)I-X(t)\widetilde P X(s)^{-1}\\
& =X(t)TT^{-1}X(s)^{-1}-X(t)\widetilde P X(s)^{-1}=X(t)(I-\widetilde P)X(s)^{-1}\\
&=X(t)\widetilde Q X(s)^{-1}.
\end{split}
\end{equation}
By~\eqref{eq:dichotomy1} and \eqref{eq:phi-t-s-P-s}, we get, for $t \ge s$,
\[
\|X(t)\widetilde P X(s)^{-1}\|=\|\Phi(t,s)P(s)\|=K \left(\frac{\mu(t)}{\mu(s)}\right)^{\alpha}\mu(s)^{\sign(s)\theta}
\]
and by~\eqref{eq:dichotomy2} and \eqref{eq:phi-t-s-Q-s}, we obtain, for $t \le s$,
\[
\|X(t)\widetilde Q X(s)^{-1}\|=\|\Phi(t,s)Q(s)\|=K\left(\frac{\mu(t)}{\mu(s)}\right)^{\beta}\mu(s)^{\sign(s)\nu},
\]
obtaining~\eqref{eq:dichotomy1-fund-matrix} and~\eqref{eq:dichotomy2-fund-matrix}.
\end{remark}

Let $\mu:\R\to\R^+$ be a differentiable growth rate. We define the {\it nonuniform $\mu$-dichotomy spectrum} of system \eqref{eq:sist-linear} by
\[
\SigmaND_\mu(A)=\left\{\gamma\in\R;\, x'=\left(A(t)-\gamma \frac{\mu'(t)}{\mu(t)}I\right)x \ \text{\small admits no N$\mu$D }\right\}
\]
and refer to the complement of this set, $\rhoND_\mu(A)=\R\setminus \SigmaND_\mu(A)$, as
{\it nonuniform $\mu$-resolvent set} of system \eqref{eq:sist-linear}.
We also define the {\it $\mu$-dichotomy spectrum} of system \eqref{eq:sist-linear} by
\[
\SigmaD_\mu(A)=\left\{\gamma\in\R;\, x'=\left(A(t)-\gamma \frac{\mu'(t)}{\mu(t)}I\right)x \ \text{ admits no $\mu$D }\right\}
\]
and call the complement of this set, denoted $\rhoD_\mu(A)=\R\setminus \SigmaD_\mu(A)$, by
{\it $\mu$-resolvent set} of system \eqref{eq:sist-linear}. Since a $\mu$-dichotomy is also a nonuniform $\mu$-dichotomy (with $\nu=\theta=0$), we have
$\Sigma_\mu^{ND}(A)\subseteq \Sigma_\mu^D(A)$. Following~\cite{Chu-Liao-Siegmund-Xia-Zhang-BSM-2015}, when $\mu=(\e^n)_{n \in \N}$ we denote the nonuniform $\mu$-dichotomy spectrum obtained by \emph{nonuniform dichotomy spectrum}. When $p$ is the growth rate in~\eqref{eq:poly-dich-R} we obtain the new notion of spectrum
\[
\SigmaND_{p}(A)
=\left\{\gamma\in\R;\, x'=\left(A(t)- \frac{\gamma}{1+|t|}I\right)x
\text{ \small admits no NPD}\right\}
\]
that we call \emph{nonuniform polynomial dichotomy spectrum}.

\begin{remark}
In the very particular case $\mu(t)=\e^t$, $\alpha=-\beta$ and $\theta=\nu$, we recover the notion of nonuniform dichotomy spectrum considered in~\cite{Chu-Liao-Siegmund-Xia-Zhang-BSM-2015}. On the other hand, to recover the notion of nonuniform dichotomy spectrum in~\cite{Zhang-JFA-2014} when $\mu(t)=\e^t$, one needs to add the extra condition $\max\{\theta,\nu\}\le\min\{-\alpha,\beta\}$ to notion of nonuniform exponential dichotomy. Thus, in the particular case of nonuniform exponential dichotomies, our spectrum is contained in the one considered in~\cite{Zhang-JFA-2014}.
\end{remark}

A nonempty set $\mathcal W$ of $\R\times\R^n$ such that $\{(t,\Phi(t,s)\xi):\,$ $t\in\R\}\subset \mathcal W$  for each
$(s,\xi)\in \mathcal W$ is called a {\it linear integral manifold} of system \eqref{eq:sist-linear}. For each $s\in\R$, the set
$$\mathcal W(s)=\{\xi\in\R^n;\,(s, \xi)\in \mathcal W\}$$
is a linear subspace of $\R^n$ called a {\it fiber} of the linear integral manifold $\mathcal W$. Note that all the fibers of a linear integral manifold have the same dimension and form a vector bundle over $\R$. We define the {\it rank} of $\mathcal W$, denoted by $\rank \mathcal W$, as the dimension of each of the fibers of $\mathcal W$. We also define the intersection and the sum of linear integral manifolds of \eqref{eq:sist-linear}, $\mathcal W_1$ and $\mathcal W_2$, respectively by
$$\mathcal W_1\cap \mathcal W_2 = \{(s,\xi)\in\R\times\R^n;\,\xi\in \mathcal W_1(s)\cap \mathcal W_2(s)\}$$
and
$$\mathcal W_1+\mathcal W_2 = \{(s,\xi)\in\R\times\R^n;\,\xi\in \mathcal W_1(s)+\mathcal W_2(s)\}.$$
If $\mathcal W_i\cap \mathcal W_j=\R\times \{0\}$ for $1\le i\ne j\le k$, the sum of the linear integral manifolds $\mathcal W_1,\ldots,\mathcal W_k$ is called the {\it Whitney sum} of those linear integral manifolds and is denoted by $\mathcal W_1\oplus\ldots\oplus \mathcal W_k$.
\section{Spectral theorems}\label{section:spectral-theorems}
Given $\gamma\in\R$, define the sets
\begin{equation}\label{eq:stable-inv-manif}
\mathcal
U_\gamma=\left\{(s,\xi)\in\R\times\R^n:\,\sup\limits_{t\ge
0}\|\Phi(t,s)\xi\|\mu(t)^{-\gamma}<\infty\right\}
\end{equation}
and
\begin{equation}\label{eq:unstable-inv-manif}
\mathcal V_\gamma=\left\{(s,\xi)\in\R\times\R^n:\,\sup\limits_{t\le
0}\|\Phi(t,s)\xi\|\mu(t)^{-\gamma}<\infty\right\}.
\end{equation}
These sets will be used in this section to discuss the topological structure of the nonuniform $\mu$-dichotomy spectrum of system \eqref{eq:sist-linear}.
Our first result includes Theorem 1.1 in~\cite{Zhang-JFA-2014} as the particular case of exponential growth rates.

\begin{theorem}\label{thm:spectrum} Let $\mu:\R\to\R^+$ be a differentiable growth rate. The following statements hold for system \eqref{eq:sist-linear}:
\begin{enumerate}[\lb=$\arabic*)$,\lm=5mm]
\item \label{thm:spectrum-1} There is an $m \in \{0,\ldots,n\}$ such that the nonuniform $\mu$-dichotomy spectrum $\SigmaND_\mu(A)$ of
system \eqref{eq:sist-linear} is the union of $m$ disjoint closed intervals
in $\R$:
\begin{enumerate}[\lb=$\alph*)$,\lm=6mm,\is=2mm]
  \item if $m=0$ then $\SigmaND_\mu(A)=\emptyset$;
  \item if $m=1$ then $$\SigmaND_\mu(A)=\mathbb R \ \text{or} \ \SigmaND_\mu(A)=(-\infty,b_1] \ \text{or} \ \SigmaND_\mu(A)=[a_1,b_1] \ \text{or} \ \SigmaND_\mu(A)=[a_1,\infty);$$
  \item \label{thm:spectrum-1c} if $1<m\le n$ then
  $$\SigmaND_\mu(A)=I_1\cup [a_2,b_2]\cup\ldots\cup [a_{m-1},b_{m-1}]\cup I_m$$
  with $I_1=[a_1,b_1]$ or $(-\infty,b_1]$, $I_m=[a_m,b_m]$ or $[a_m,\infty)$ and $a_i\le b_i<a_{i+1}$ for $i=1,\ldots,m-1$.
\end{enumerate}
\item \label{thm:spectrum-2} Assume $m \ge 1$, write
    $$\SigmaND_\mu(A)=I_1\cup [a_2,b_2]\cup\ldots\cup [a_{m-1},b_{m-1}]\cup I_m.$$
and, for $i=0,\dots,m+1$, define
\[
\mathcal W_i=
\begin{cases}
\mathbb R\times\{0\}  &\quad \text{ if } \ i=0 \ \text{ and } \ I_1=(-\infty,b_1]\\
\mathcal U_{\gamma_0} \ \text{ for some } \gamma_0\in(-\infty,a_1) &\quad \text{ if } \ i=0 \ \text{ and } \ I_1=[a_1,b_1]\\
\mathcal U_{\gamma_i} \cap \mathcal V_{\gamma_{i-1}}  \ \text{ for some } \gamma_i \in(b_i,a_{i+1}) &\quad \text{ if } \ i=1,\ldots,m\\
\mathcal V_{\gamma_m} \ \text{ for some } \gamma_m \in(b_m,+\infty) &\quad \text{ if } \ i=m+1 \ \text{ and } \ I_m=[a_m,b_m]\\
\mathbb R\times\{0\} \ &\quad \text{ if } \ i=m+1 \ \text{ and } \ I_m=[a_m,+\infty)\\
\end{cases}.
\]
Then, the sets $\mathcal U_{\gamma_i}$ $\mathcal V_{\gamma_i}$ and $\mathcal W_i$, $i=0,\ldots,m+1$, are integral manifolds, $\rank\mathcal W_i\ge 1$ for $i=1,\dots,m$ and
\[
\mathcal W_0\oplus\mathcal W_1\oplus\ldots\oplus\mathcal W_{m+1}=\R\times\R^n.
\]
\end{enumerate}
\end{theorem}

We call the intervals in $\SigmaND$ spectral intervals and each linear integral manifold $\mathcal W_i$, $i=0,\dots,m+1$, a spectral manifold.

\begin{proof}

To prove our result we begin by establishing some lemmas.

\begin{lemma}\label{lemma:linear-manif}
The sets $\mathcal U_\gamma$ in~\eqref{eq:stable-inv-manif} and $\mathcal V_\gamma$ in~\eqref{eq:unstable-inv-manif} are
linear integral manifolds of system \eqref{eq:sist-linear}. Additionally, if $\gamma\le \widetilde\gamma$ then $\mathcal
U_{\gamma}\subseteq \mathcal U_{\widetilde\gamma}$ and $\mathcal V_{\widetilde\gamma}\subseteq \mathcal V_{\gamma}$.
\end{lemma}

\begin{proof} Since
\[
\sup\limits_{t\ge 0}\|\Phi(t,\tau)\Phi(\tau,s)\xi\|\mu(t)^{-\gamma}=
\sup\limits_{t\ge 0}\|\Phi(t,s)\xi\|\mu(t)^{-\gamma}<\infty,
\]
we conclude that for any $(s,\xi)\in\mathcal U_\gamma$ we have $(\tau,\Phi(\tau,s)\xi)\in\mathcal
U_\gamma$ for all $\tau\in\R$. Thus $\mathcal U_\gamma$ is a linear integral manifold of system \eqref{eq:sist-linear}.
Using a similar argument we conclude that $\mathcal V_{\gamma}$ is also a linear integral manifold of system \eqref{eq:sist-linear}.

Since $\mu$ is increasing, $\gamma\le \widetilde\gamma$ implies $\mu(t)^{-\widetilde\gamma}\ge\mu(t)^{-\gamma}$ and we have immediately $\mathcal U_{\gamma}\subseteq\mathcal U_{\widetilde\gamma}$ and $\mathcal V_{\gamma}\supseteq \mathcal V_{\widetilde\gamma}$.
\end{proof}

\begin{lemma}\label{prop:exist-dich}
Let $\gamma\in\R$. If $x'=\left(A(t)-\gamma \frac{\mu'(t)}{\mu(t)} I\right)x$ admits a nonuniform $\mu$-dichotomy with the invariant projection $P$, then we have
$\mathcal U_\gamma=\mbox{\rm Im} P$, $\mathcal V_\gamma=\mbox{\rm Ker} P$ and $\mathcal U_\gamma\oplus \mathcal V_\gamma=\R\times \R^n$.
\end{lemma}

\begin{proof}
Let $\Phi(t,s)$ be the evolution operator of the linear equation \mbox{$x'=A(t)x$} and
\begin{equation}\label{eq:Phi-gamma}
\Phi_\gamma(t,s)=\left(\frac{\mu(t)}{\mu(s)}\right)^{-\gamma}\Phi(t,s).
\end{equation}
Noting that
\[
\begin{split}
\Phi_\gamma(t,s)'
& =-\gamma\left(\frac{\mu(t)}{\mu(s)}\right)^{-\gamma-1}\frac{\mu'(t)}{\mu(s)}\Phi(t,s)+\left(\frac{\mu(t)}{\mu(s)}\right)^{-\gamma}\Phi(t,s)'\\
& =\left(A(t)-\gamma\frac{\mu'(t)}{\mu(t)}I\right)\left(\frac{\mu(t)}{\mu(s)}\right)^{-\gamma}\Phi(t,s)\\
& =\left(A(t)-\gamma\frac{\mu'(t)}{\mu(t)}I\right)\Phi_\gamma(t,s)\\
\end{split}
\]
and $\Phi_\gamma(s,s)=\left(\frac{\mu(s)}{\mu(s)}\right)^{-\gamma}\Phi(s,s)=I$, we conclude that $\Phi_\gamma(t,s)$ is the evolution operator of the linear equation $x'=\left(A(t)-\gamma \frac{\mu'(t)}{\mu(t)} I\right)x$. Moreover, the projection $P$ still commutes with $\Phi_\gamma(t,s)$:
\[
\Phi_\gamma(t,s)P(s)=\left(\frac{\mu(t)}{\mu(s)}\right)^{-\gamma}\Phi(t,s)P(s)=\left(\frac{\mu(t)}{\mu(s)}\right)^{-\gamma}P(t)\Phi(t,s)=P(t)\Phi_\gamma(t,s).
\]

By the assumption, there exist
$K_\gamma\ge 1$, $\alpha_\gamma<0,\beta_\gamma>0$ and
$\theta_\gamma,\nu_\gamma\ge 0$ with $\alpha+\theta_\gamma<0$ and
$\beta-\nu_\gamma>0$ such that
\begin{equation}\label{eq:dich-phi-gamma}
\begin{split}
\|\Phi_\gamma(t,s)P(s)\|&\le K_\gamma
\left(\frac{\mu(t)}{\mu(s)}\right)^{\alpha_\gamma}\mu(s)^{\sign(s)\theta_\gamma} \ \ \mbox{ for all } \ \ t\ge s,\\
\|\Phi_\gamma(t,s)Q(s)\|&\le K_\gamma
\left(\frac{\mu(t)}{\mu(s)}\right)^{\beta_\gamma}\mu(s)^{\sign(s)\nu_\gamma} \ \ \mbox{ for all }  \ \ t\le s.
\end{split}
\end{equation}

First we prove $\mathcal U_\gamma\subseteq{\rm Im} P$.  Let
$(\tau,\xi)\in\mathcal U_\gamma$. By definition there exists a
constant $c_\gamma$ such that
\[
\|\Phi(t,\tau)\xi\|\le c_\gamma \mu(t)^\gamma \mbox{ for all } t\ge
0.
\]
It follows that
\begin{equation}\label{eq:Phittau_cgammamugamma}
\|\Phi_\gamma(t,\tau)\xi\|=\left(\frac{\mu(t)}{\mu(\tau)}\right)^{-\gamma}\|\Phi(t,\tau)\xi\|\le
c_\gamma \mu(\tau)^\gamma \mbox{ for all } t\ge 0.
\end{equation}
Let $\xi=\xi_1+\xi_2$ with $\xi_1\in{\rm Im}P(\tau)$ and
$\xi_2\in{\rm Ker}P(\tau)$. We have
\[
\begin{split}
\xi_2
& =(I-P(\tau))\xi=(I-\Phi_\gamma(\tau,t)\Phi_\gamma(t,\tau)P(\tau))\xi\\
& =(I-\Phi_\gamma(\tau,t)P(t)\Phi_\gamma(t,\tau))\xi=\Phi_\gamma(\tau,t)Q(t)\Phi_\gamma(t,\tau)\xi.
\end{split}
\]
Hence, using~\eqref{eq:Phittau_cgammamugamma}, we have, for $t\ge\tau\ge 0$,
\[
\|\xi_2\| \le K_\gamma \left(\frac{\mu(\tau)}{\mu(t)}\right)^{\beta_\gamma}\mu(t)^{\nu_\gamma}\|\Phi_\gamma(t,\tau)\xi\|
\le K_\gamma c_\gamma \mu(t)^{-(\beta_\gamma-\nu_\gamma)}\mu(\tau)^{\beta_\gamma+\gamma}.
\]
Therefore, $\xi_2=0$ because $\beta_\gamma-\nu_\gamma>0$. Thus $\xi=\xi_1\in{\rm Im}P(\tau)$. This proves that
$\mathcal U_\gamma\subseteq{\rm Im}P$.

To establish that ${\rm Im} P\subseteq \mathcal U_\gamma$, assume that
$\tau\in\R$, $\xi\in{\rm Im}P(\tau)$. Then $P(\tau)\xi=\xi$. By~\eqref{eq:dich-phi-gamma}, we have, for $t\ge\max\{\tau, 0\}$,
\[
\begin{split}
\|\Phi(t,\tau)\xi\|\mu(t)^{-\gamma}
& =\|\Phi_\gamma(t,\tau)P(\tau)\xi\|\mu(\tau)^{-\gamma}\\\
& \le K_\gamma
\mu(t)^{\alpha_\gamma}\mu(\tau)^{-(\gamma+\alpha_\gamma)}\mu(\tau)^{\sign(\tau)\theta_\gamma}\|\xi\|.
\end{split}
\]
This implies that $(\tau,\xi)\in \mathcal U_\gamma$, because
$\alpha_\gamma<0$. Hence ${\rm Im}P\subset \mathcal U_\gamma$.
We obtain ${\rm Im}P= \mathcal U_\gamma$.

Using $\alpha_\gamma+\theta_\gamma<0$, we can apply similar arguments to
prove that $\mathcal V_\gamma={\rm Ker}P$. Since $\mathcal U_\gamma={\rm Im}P$ and $\mathcal V_\gamma={\rm
Ker}P$, the identity $\mathcal U_\gamma\oplus\mathcal V_\gamma=\R\times\R^n$ follows.
\end{proof}

\begin{lemma}\label{lemma:resolvent}
The resolvent set $\rhoND_\mu(A)$ is open. Moreover, if $\gamma\in\rhoND_\mu(A)$ and
$J\subset \rhoND_\mu(A)$ is an interval containing $\gamma$, then
\[
\mathcal U_\eta=\mathcal U_\gamma,\quad \mathcal V_\eta=\mathcal
V_\gamma \quad \mbox{ for all } \eta\in J.
\]
\end{lemma}

\begin{proof}
Let $\gamma\in\rhoND_\mu(A)$. Then $x'=\left(A(t)-\gamma \frac{\mu'(t)}{\mu(t)} I\right)x$
admits a N$\mu$D: there is some family of projections $P(t)$ and constants \mbox{$K\ge 1$}, $\alpha<0$, $\beta>0$
and $\theta,\nu\ge 0$, with $\alpha+\theta<0$ and $\beta-\nu>0$ such that
\begin{eqnarray*}
\|\Phi_\gamma(t,s)P(s)\|&\le& K \left(\frac{\mu(t)}{\mu(s)}\right)^{\alpha}\mu(s)^{\sign(s)\theta} \mbox{ for } t\ge s,\\
\|\Phi_\gamma(t,s)Q(s)\|&\le& K \left(\frac{\mu(t)}{\mu(s)}\right)^{\beta}\mu(s)^{\sign(s)\nu} \mbox{
for } t\le s.
\end{eqnarray*}
Let $0<\sigma<\min\{(\beta-\nu)/2,-(\alpha+\theta)/2\}$. It follow from the proof of Lemma~\ref{prop:exist-dich}, that
$P(t)$ is an invariant projection for the evolution operator
$\Phi_\eta(t,s)=\left(\frac{\mu(s)}{\mu(t)}\right)^\eta\Phi(t,s)$ of system $x'=\left(A(t)-\eta\frac{\mu'(t)}{\mu(t)} I\right)x$. Additionally, for
$\eta\in(\gamma-\sigma,\gamma+\sigma)$ we have
\[
\begin{split}
\|\Phi_\eta(t,s)P(s)\|&=\left(\frac{\mu(s)}{\mu(t)}\right)^{\eta-\gamma}\|\Phi_\gamma(t,s)P(s)\|\\
& \le K \left(\frac{\mu(t)}{\mu(s)}\right)^{\alpha+\gamma-\eta}\mu(s)^{\sign(s)\theta} \,\,\mbox{ for } t\ge s,\\
\|\Phi_\eta(t,s)Q(s)\|&=\left(\frac{\mu(s)}{\mu(t)}\right)^{\eta-\gamma}\|\Phi_\gamma(t,s)Q(s)\|\\
& \le K \left(\frac{\mu(t)}{\mu(s)}\right)^{\beta+\gamma-\eta}\mu(s)^{\sign(s)\nu} \,\,\mbox{ for } t\le s.
\end{split}
\]
This proves that
$x'=\left(A(t)-\eta \frac{\mu'(t)}{\mu(t)} I\right)x$ admits a N$\mu$D
for all $\eta\in (\gamma-\sigma,\gamma+\sigma)$. Consequently $(\gamma-\sigma,\gamma+\sigma)\subset \rhoND_\mu(A)$, for sufficiently small $\sigma>0$. Hence,
$\rhoND_\mu(A)$ is an open set. We also conclude that systems $x'=\left(A(t)-\eta \frac{\mu'(t)}{\mu(t)} I\right)x$ and $x'=\left(A(t)-\gamma \frac{\mu'(t)}{\mu(t)}I\right)x$
admit a N$\mu$D with the same
projection, $P(t)$, when $\eta\in(\gamma-\sigma,\gamma+\sigma)$. According to Lemma \ref{prop:exist-dich},
we have $\mathcal U_\eta=\mathcal U_\gamma={\rm Im}P$ and $\mathcal
V_\eta=\mathcal V_\gamma={\rm Ker}P$.

Let $J \subset \rhoND_\mu(A)$ be an interval containing $\gamma$ and $\gamma_0$. Assuming $\gamma_0\le\gamma$, we have $[\gamma_0,\gamma] \subset J$. For each $a\in
[\gamma_0,\gamma]$ there exists $\sigma_a>0$
such that $\mathcal U_\zeta=\mathcal U_a$ and
$\mathcal V_\zeta=\mathcal V_a$ for all
$\zeta\in(a-\sigma_a,a+\sigma_a)$. Since these open
intervals cover $[\gamma_0,\gamma]$, we get that  $\mathcal
U_\gamma=\mathcal U_{\gamma_0}$ and $\mathcal V_\gamma=\mathcal
V_{\gamma_0}$. The same argument leads to the conclusion that the same property holds when $\gamma_0>\gamma$.
Since $\gamma_0\in J$ is arbitrary, we obtain the result.
\end{proof}

The next lemma characterizes the intersection of the linear integral manifolds
$\mathcal U_{\gamma_2}$ and $\mathcal V_{\gamma_1}$ with $\gamma_1,\gamma_2\in\rhoND_\mu(A)$.

\begin{lemma}\label{prop:equiv}
Let $\gamma_1,\gamma_2\in\rhoND_\mu(A)$ and $\gamma_1<\gamma_2$. The
following conditions are equivalent.
\begin{enumerate}[\lb=$\alph*)$,\lm=6mm,\is=2mm]
\item\label{prop:equiv-1} $\mathcal U_{\gamma_2}\cap\mathcal V_{\gamma_1} \ne \R\times \{0\}$;
\item\label{prop:equiv-2} $[\gamma_1,\gamma_2]\cap\SigmaND_\mu(A)\ne \emptyset$;
\item\label{prop:equiv-3} $\rank\,\mathcal U_{\gamma_1}<\rank\,\mathcal U_{\gamma_2}$;
\item\label{prop:equiv-4} $\rank\,\mathcal V_{\gamma_1}>\rank\,\mathcal V_{\gamma_2}$.
\end{enumerate}
\end{lemma}

\begin{proof}
By Lemma~\ref{prop:exist-dich}, we have
$$\rank\,\mathcal U_{\gamma_1} + \rank\,\mathcal V_{\gamma_1}=\rank\,\mathcal U_{\gamma_2} + \rank\,\mathcal V_{\gamma_2}=n$$
and we immediately conclude the equivalence between~\ref{prop:equiv-3} and~\ref{prop:equiv-4}.

Assume that \ref{prop:equiv-3} holds. By Lemma~\ref{lemma:linear-manif}, taking into account that $\mathcal U_{\gamma_1}\subseteq \mathcal U_{\gamma_2}$, we have $\mathcal U_{\gamma_1}\cup\mathcal V_{\gamma_1} \subseteq \mathcal U_{\gamma_2}\cup\mathcal V_{\gamma_1} \subseteq\R\times\R^n$ and therefore
\[
\rank\,(\mathcal U_{\gamma_2}\cap\mathcal
V_{\gamma_1})\ge\rank\,\mathcal U_{\gamma_2}+\rank\,\mathcal
V_{\gamma_1}-n>\rank\,\mathcal U_{\gamma_1}+\rank\,\mathcal
V_{\gamma_1}-n=0.
\]
So $\mathcal U_{\gamma_2}\cap\mathcal V_{\gamma_1} \ne \R\times\{0\}$, which proves \ref{prop:equiv-1}.

Assume that \ref{prop:equiv-1} holds and assume, by contradiction, that $[\gamma_1,\gamma_2]\cap\SigmaND_\mu(A)= \emptyset$ and therefore that $[\gamma_1,\gamma_2]\subset\rhoND_\mu(A)$. It follows from Lemma~\ref{lemma:resolvent} that
\[
\mathcal U_{\gamma_2}\cap\mathcal V_{\gamma_1}= \mathcal
U_{\gamma_1}\cap\mathcal V_{\gamma_1}=\R\times\{0\},
\]
a contradiction with \ref{prop:equiv-1}. We conclude that \ref{prop:equiv-2} holds.

Next we assume, by contradiction, that \ref{prop:equiv-2} holds but \ref{prop:equiv-3} doesn't hold, that is, \ref{prop:equiv-2} holds but $\rank\,\mathcal U_{\gamma_1}\ge \rank\,\mathcal U_{\gamma_2}$. Since $\gamma_1 < \gamma_2$, by Lemma~\ref{lemma:linear-manif}, we must have $\rank\,\mathcal U_{\gamma_1}\le \rank\,\mathcal U_{\gamma_2}$. Thus $\rank\,\mathcal U_{\gamma_1}=\rank\,\mathcal U_{\gamma_2}$. By the
equivalence of \ref{prop:equiv-3} and \ref{prop:equiv-4} we also have $\rank \mathcal
V_{\gamma_1}=\rank \mathcal V_{\gamma_2}$. Since $\mathcal U_{\gamma_i}(t)$ and $\mathcal
V_{\gamma_i}(t)$, $i=1,2$,  are linear subspaces of $\R^{n}$, we must have $\mathcal U_{\gamma_1}(t)=\mathcal U_{\gamma_2}(t)$ and $\mathcal V_{\gamma_1}(t)=\mathcal V_{\gamma_2}(t)$. Thus $\mathcal U_{\gamma_1}=\mathcal U_{\gamma_2}$ and $\mathcal V_{\gamma_1}=\mathcal V_{\gamma_2}$.
As a consequence, by Lemma~\ref{prop:exist-dich}, the nonuniform $\mu$-dichotomies of $x'=\left(A(t)-\gamma_1 \frac{\mu'(t)}{\mu(t)}I\right)x$ and $x'=\left(A(t)-\gamma_2 \frac{\mu'(t)}{\mu(t)}I\right)x$ have the same invariant projection $P(t)$. Let $K_i\ge 1$, $\alpha_i<0$, $\beta_i>0$, $\theta_i\ge0$ and $\nu_i\ge 0$, satisfying $\alpha_i+\theta_i<0$ and $\beta_i-\nu_i>0$, for $i=1,2$, be such that
\[
\|\Phi_{\gamma_i}(t,s)P(s)\|\le  K_i\left(\frac{\mu(t)}{\mu(s)}\right)^{\alpha_i} \mu(s)^{\sign(s)\theta_i} \,\mbox{ for } t\ge s
\]
and
\[
\|\Phi_{\gamma_i}(t,s)Q(s)\|\le K_i\left(\frac{\mu(t)}{\mu(s)}\right)^{\beta_i} \mu(s)^{\sign(s)\nu_i}
\,\mbox{ for } t\le s.
\]

For $\gamma\in[\gamma_1,\gamma_2]$, we have for $t\ge s$
\begin{equation}\label{eq:sei-la-1}
\begin{split}
\|\Phi_{\gamma}(t,s)P(s)\|
& =\left(\frac{\mu(t)}{\mu(s)}\right)^{\gamma_1-\gamma}\|\Phi_{\gamma_1}(t,s)P(s)\|\\
& \le \|\Phi_{\gamma_1}(t,s)P(s)\|\\
& \le
K\left(\frac{\mu(t)}{\mu(s)}\right)^{\alpha_1} \mu(s)^{\sign(s)\theta_1}
\end{split}
\end{equation}
and, for $ t\le s$,
\begin{equation}\label{eq:sei-la-2}
\begin{split}
\|\Phi_{\gamma}(t,s)Q(s)\|
& =\left(\frac{\mu(t)}{\mu(s)}\right)^{\gamma_2-\gamma}\|\Phi_{\gamma_2}(t,s)Q(s)\|\\
& \le\|\Phi_{\gamma_2}(t,s)Q(s)\|\\
& \le K\left(\frac{\mu(t)}{\mu(s)}\right)^{\beta_2} \mu(s)^{\sign(s)\nu_2},
\end{split}
\end{equation}
where $K=\max\{K_1,K_2\}$. Since $\alpha_1+\theta_1<0$ and $\beta_2-\nu_2>0$, we conclude that $\Phi_\gamma$ admits a N$\mu$D with constants $\alpha_1$, $\beta_2$, $\theta_1$, $\nu_2$ and
$K$. Thus $\gamma\in\rhoND_\mu(A)$ and consequently
$[\gamma_1,\gamma_2]\subset\rhoND_\mu(A)$. This contradicts~\ref{prop:equiv-2}. We conclude that \ref{prop:equiv-3} hold and the lemma is established.
\end{proof}

Now we will use the auxiliary results obtained above to prove the theorem. According to Lemma \ref{lemma:resolvent}, $\rhoND_\mu(A)$ is a nonempty open subset of $\R$ and therefore it can be written as a finite or countable union of open mutually disjoint intervals. Thus, $\SigmaND_\mu(A)=\R\setminus \rhoND_\mu(A)$ is either empty or consists of a finite or countable union of closed intervals with vanishing intersection. Let $m \in \N \cup \{+\infty\}$ be the number of disjoint closed intervals whose union is $\SigmaND_\mu$:
    $$\SigmaND_\mu(A)= \cdots \cup [a_{i-1},b_{i-1}] \cup [a_i,b_i] \cup \cdots,$$
with $a_{i-1}\le b_{i-1}<a_i\le b_i$.  Assume, by contradiction, that $m>n$. Choose $n+1$ consecutive disjoint intervals, $[a_{k_0},b_{k_0}],\ldots,[a_{{k_0}+n},b_{{k_0}+n}]$, and, for $j=k_0,\ldots,k_0+n-1$, let $\gamma_j \in (b_j,a_{j+1})$. By the equivalence of~\ref{prop:equiv-1} and \ref{prop:equiv-3} in Lemma~\ref{prop:equiv}, we have
$$0\le\rank\,\mathcal U_{\gamma_{k_0}}<\rank\,\mathcal U_{\gamma_{k_0+1}}<\ldots<\rank\,\mathcal U_{\gamma_{k_0+n}}\le n,$$
which allows us to conclude that $\rank\,\mathcal U_{\gamma_{k_0+n}}=n$ or $\rank\,\mathcal U_{\gamma_{k_0}}=0$.

If $\rank\,\mathcal U_{\gamma_{k_0+n}}=n$, we must have $P=I$, $\mathcal U_{\gamma_{k_0+n}}=\R\times\R^n$ and $\mathcal V_{\gamma_{k_0+n}}=\R\times\{0\}$. By definition, $x'=\left(A(t)-\gamma \frac{\mu'(t)}{\mu(t)} I\right)x$ admits a N$\mu$D with invariant projection $P=I$ for $\gamma = \gamma_{k_0+n}$. Thus, for $\gamma > \gamma_{k_0+n}$ and $t\ge s$ we have
$$\|\Phi_\gamma(t,s)\|=\left(\frac{\mu(t)}{\mu(s)}\right)^{\gamma_{k_0+n}-\gamma}\left\|\Phi_{\gamma_{k_0+n}}(t,s)\right\|\le
K\left(\frac{\mu(t)}{\mu(s)}\right)^{\gamma_{k_0+n}-\gamma+\alpha} \mu(s)^{\sign(s)\theta},$$
for some $K\ge 1$, $\alpha<0$ and $\theta\ge0$, with $\alpha+\theta<0$ (which implies that $\gamma_{k_0+n}-\gamma+\alpha+\theta<0$). This shows that $x'=\left(A(t)-\gamma \frac{\mu'(t)}{\mu(t)} I\right)x$ admits a N$\mu$D with the invariant projection $P=I$ for all $\gamma > \gamma_{k_0+n}$.

If $\rank\,\mathcal U_{\gamma_{k_0}}=0$, we must have $P=0$, $\mathcal U_{\gamma_{k_0}}=\R\times\{0\}$ and $\mathcal V_{\gamma_{k_0}}=\R\times\R^n$. By definition, $x'=\left(A(t)-\gamma \frac{\mu'(t)}{\mu(t)} I\right)x$ admits a N$\mu$D with the invariant projection $P=0$ for $\gamma = \gamma_{k_0}$.
Thus, for $\gamma < \gamma_{k_0}$ and $t\le s$ we have
$$\|\Phi_\gamma(t,s)\|=\left(\frac{\mu(t)}{\mu(s)}\right)^{\gamma_{k_0}-\gamma}\|\Phi_{\gamma_{k_0}}(t,s)\|
\le K\left(\frac{\mu(t)}{\mu(s)}\right)^{\gamma_{k_0}-\gamma+\beta} \mu(s)^{\sign(s)\nu},$$
for some $K\ge 1$, $\beta>0$ and $\nu\ge0$, with $\beta-\nu>0$ (which implies that $\gamma_{k_0}-\gamma+\beta-\mu>0$). This shows that $x'=\left(A(t)-\gamma \frac{\mu'(t)}{\mu(t)} I\right)x$ admits a N$\mu$D with the invariant projection $P=0$ for all $\gamma < \gamma_{k_0}$.

This is in contradiction with the assumption that $m>n$. We obtain~\ref{thm:spectrum-1}.

We now prove~\ref{thm:spectrum-2}. We begin by noting that, by Lemma~\ref{prop:exist-dich}, the spectral manifolds $\mathcal W_i$ are independent of the choice of $\gamma_i$.

Assuming that $I_1=[a_1,b_1]$, take $\gamma_0 \in (-\infty,a_1)$ and $\gamma_1\in(b_1,a_2)$. Then $\gamma_0,
\gamma_1 \in \rhoND_\mu(A)$ and
$[\gamma_0,\gamma_1]\cap\SigmaND_\mu(A)\ne\emptyset$. By Lemmas~\ref{lemma:linear-manif} and
\ref{prop:equiv} we have $\mathcal U_{\gamma_0}\subsetneq\mathcal
U_{\gamma_1}$. Since $\mathcal U_{\gamma_0}\oplus\mathcal
V_{\gamma_0}=\R\times\R^n$, we must have $\mathcal
U_{\gamma_0}\cap\mathcal V_{\gamma_0} \subsetneq \mathcal
U_{\gamma_1}\cap\mathcal V_{\gamma_0} = \mathcal W_1$. Since $\mathcal W_1$ is a linear
integral manifold, we conclude that $\rank\mathcal W_1\ge 1$ .

If $I_1= (-\infty,b_1]$ then $\mathcal U_{\gamma_0}=\R\times\{0\}$ which implies $\mathcal V_{\gamma_0}=\R\times\R^n$. Thus $\mathcal W_1=\mathcal U_{\gamma_1}\cap\mathcal V_{\gamma_0}=\mathcal U_{\gamma_1}$. By contradiction, if $\rank\mathcal W_1=0$ we would have $\mathcal W_1=\R\times
\{0\}$ and thus we would have a
N$\mu$D for $x'=\left(A(t)-\gamma_1 \frac{\mu'(t)}{\mu(t)}I\right)x$ with the invariant projection $P\equiv 0$. Proceeding as in the proof of \ref{thm:spectrum-1}, we conclude that
$(-\infty,\gamma_1]\subset\rhoND_\mu(A)$, a contradiction with
the choice of $\gamma_1$. We conclude that $\rank\mathcal W_1\ge 1$ also in this case.

For $i=\{1,\ldots,m\}$, we have $\gamma_{i-1},\gamma_i\in\rhoND_\mu(A)$ and
$[\gamma_{i-1},\gamma_i]\cap\SigmaND_\mu(A)\ne\emptyset$. Thus, by Lemma~\ref{prop:equiv}, we have
$\mathcal U_{\gamma_i}\cap\mathcal V_{\gamma_{i-1}}=\mathcal W_i \ne \R \times \{0\}$ and consequently
$\rank\mathcal W_i\ge 1$. Additionally, since $\mathcal V_{\gamma_{i+1}}\subset\mathcal V_{\gamma_{i}}$, we have, for $i=0,\ldots,m-1$,
\[
\begin{split}
\mathcal W_{i+1}+\mathcal V_{\gamma_{i+1}}
& = \{(\tau,\eta) \in \R\times\R^n: \eta \in \mathcal W_{i+1}(\tau)+\mathcal V_{\gamma_i}(\tau)\}\\
& = \{(\tau,\eta) \in \R\times\R^n: \eta \in \mathcal V_{\gamma_i}(\tau) \cap \mathcal U_{\gamma_{i+1}}(\tau)+\mathcal V_{\gamma_i}(\tau)\}\\
& = \{(\tau,\eta) \in \R\times\R^n: \eta \in \mathcal V_{\gamma_i}(\tau) \cap \left(\mathcal U_{\gamma_{i+1}}(\tau)+\mathcal V_{\gamma_{i+1}}(\tau)\right)\}\\
& = \{(\tau,\eta) \in \R\times\R^n: \eta \in \mathcal V_{\gamma_i}(\tau)\} =\mathcal V_{\gamma_i}.
\end{split}
\]
As a consequence, we have
\[
\begin{split}
\mathcal W_{0}+\mathcal W_1+\ldots+\mathcal
W_{m}+\mathcal W_{m+1}
& = \mathcal W_{0}+\mathcal W_1+\ldots+\mathcal W_{m}+\mathcal
V_{\gamma_m}\\
& = \mathcal W_{0}+\mathcal W_1+\ldots+\mathcal W_{m-1}+\mathcal V_{\gamma_{m-1}}\\
& =\ldots\\
& = \mathcal W_{0}+\mathcal W_1+\mathcal V_{\gamma_1}=\mathcal W_{0}+\mathcal V_{\gamma_0}=\mathcal U_{\gamma_0}+\mathcal V_{\gamma_0} =\R\times \R^n.
\end{split}
\]
By Lemma~\ref{lemma:linear-manif}, we have
$$\mathcal W_i\cap\mathcal W_j\subseteq\mathcal U_{\gamma_i}\cap\mathcal V_{\gamma_{j-1}}\subseteq\mathcal U_{\gamma_i}\cap\mathcal V_{\gamma_{i}}= \R\times\{0\},$$
for $0\le i<j\le m+1$.
This proves that $\R\times\R^n=\mathcal W_0\oplus\mathcal W_1\oplus \ldots\oplus\mathcal W_{m+1}$. We conclude that~\ref{thm:spectrum-2} holds
and the theorem is proved.
\end{proof}

Given $\eps\ge 0$ and a growth rate $\mu$, we say that the evolution operator $\Phi(t,s)$ of $x'=A(t)x$ has {\it nonuniformly bounded growth with respect to $(\mu,\eps)$} if there are $K \ge 1$ and $a\ge 0$ such that
\begin{equation}\label{eq:bounded-growth}
\|\Phi(t,s)\|\le K\left(\frac{\mu(t)}{\mu(s)}\right)^{\text{sgn}(t-s) a} \mu(s)^{\sign(s)\eps},\quad t,s\in\R.
\end{equation}
When $\eps=0$ the evolution operator is said to have {\it bounded growth
with respect to $\mu$}.

The next result shows that nonuniformly bounded growth with respect to the growth rate $\mu$ is a sufficient condition for the N$\mu$D spectrum to be nonempty and bounded.

\begin{theorem}\label{thm:bounded}
If the evolution operator of system \eqref{eq:sist-linear} has nonuniformly
bounded growth with respect to the growth rate $\mu$ then the nonuniform $\mu$-dichotomy spectrum $\SigmaND_\mu(A)$ of
system \eqref{eq:sist-linear} is $\SigmaND_\mu(A)=[a_1,b_1]\cup\ldots\cup [a_m,b_m]$ with $a_i,b_i \in \R$, for some $i\in\{1,\ldots,m\}$, and $b_i < a_{i+1}$, for $i=1,\ldots,m-1$.
\end{theorem}

\begin{proof}
By the assumption the evolution operator
$\Phi(t,s)$ of  system \eqref{eq:sist-linear} has a nonuniformly bounded
growth with respect to the growth rate $\mu$: there are $K \ge 1$, $a\ge 0$ and
$\eps\ge 0$ such that~\eqref{eq:bounded-growth} holds. By~\eqref{eq:bounded-growth}, for $t\ge s$ we have
\[
\|\Phi_\gamma(t,s)\|\le \left(\frac{\mu(t)}{\mu(s)}\right)^{-\gamma}\|\Phi(t,s)\|\le K\left(\frac{\mu(t)}{\mu(s)}\right)^{-\gamma+a} \mu(s)^{\sign(s)\eps}.
\]
For $\gamma>a+\eps \ \Leftrightarrow \ -\gamma+a+\eps< 0$, system $x'=\left(A(t)-\gamma \frac{\mu'(t)}{\mu(t)} I\right)x$
admits a N$\mu$D with the invariant projection $P=I$. We conclude that $(a+\eps,+\infty) \in \rhoND_\mu(A)$.
Again by~\eqref{eq:bounded-growth}, for $t\le s$ we have
\[
\|\Phi_\gamma(t,s)\|\le \left(\frac{\mu(t)}{\mu(s)}\right)^{-\gamma}\|\Phi(t,s)\|\le K\left(\frac{\mu(t)}{\mu(s)}\right)^{-\gamma-a} \mu(s)^{\sign(s)\eps}.
\]
For $\gamma<-a-\eps \ \Leftrightarrow \ -\gamma-a-\eps> 0$, system $x'=\left(A(t)-\gamma \frac{\mu'(t)}{\mu(t)} I\right)x$
admits a N$\mu$D with the invariant projection $P\equiv 0$. We conclude that
$(-\infty,-a-\eps)\subset\rhoND_\mu(A)$. Finally, since
$$\SigmaND_\mu(A)=\R\setminus \rhoND_\mu(A)\subseteq \R\setminus\{(-\infty,-a-\eps)\cup(a+\eps,+\infty)\}=[-a-\eps,a+\eps],$$
we conclude that $\SigmaND_\mu(A)$ is bounded.

By the above proof, for $\gamma>a+\eps$, we have $\mathcal U_\gamma={\rm
Im}\,I=\R\times\R^n$ and $\mathcal V_\gamma={\rm
Ker}\,I=\R\times\{0\}$ and for $\gamma<-a-\eps$, we have $\mathcal U_\gamma={\rm Im}\,0=\R\times\{0\}$ and
$\mathcal V_\gamma={\rm Ker}\,0=\R\times\R^n$. Let
\[
\gamma^*=\sup\{\gamma\in\rhoND_\mu(A):\,\mathcal
V_\gamma=\R\times\R^n\}.
\]
We have $\gamma^*\in[-a-\eps,a+\eps]$. If $\gamma^*$ were in $\rhoND_\mu(A)$, by Lemma~\ref{lemma:resolvent}, there
would be $\delta>0$ such that $(\gamma^*-\delta,\gamma^*+\delta) \, \subset
\rhoND_\mu(A)$ and thus we would have $\mathcal V_\gamma=\mathcal V_{\gamma^*}$ for any $\gamma\in \,(\gamma^*-\delta,\gamma^*+\delta)$, contradicting
the definition of $\gamma^*$. Therefore, $\gamma^*\in\SigmaND_\mu(A)$ and we conclude that $\SigmaND_\mu(A)\ne\emptyset$. Taking into account the conclusions of Theorem~\ref{thm:spectrum}, the result is established.
\end{proof}

Next, present a family of simple examples of nonautonomous systems of linear differential equations for which we are able to compute the nonuniform $\mu$-dichotomy spectrum.

\begin{example}\label{example1}
Let $\alpha<0$, $\beta >0$, $\theta,\nu\ge0$ with $\alpha+\theta<0$ and $\beta-\nu>0$ and let $\mu:\R\to\R^+$ be a differentiable growth rate. Define
\small{
\[
A(t)=
\left[
\begin{array}{cc}
a_1(t) & 0 \\
0 & a_2(t)\\
\end{array}
\right].
\]}
where
$$a_1(t)=\alpha\frac{\mu^{\prime}(t)}{\mu(t)}+\theta\sign(t)\left(\frac{\mu^{\prime}(t)}{\mu(t)} \frac{\cos t-1}{2}-\log \mu(t) \frac{\sin t}{2}\right)$$
and
$$a_2(t)=\beta\frac{\mu^{\prime}(t)}{\mu(t)}+\nu\sign(t)\left(\frac{\mu^{\prime}(t)}{\mu(t)} \frac{\cos t-1}{2}-\log \mu(t) \frac{\sin t}{2}\right).$$
Consider the differential equation in $\mathbb{R}^{2}$ given by
\begin{equation}\label{eq:exemplo}
\left[
\begin{array}{c}
u'\\ v'
\end{array}
\right]
=
A(t)
\left[
\begin{array}{c}
u\\ v
\end{array}
\right].
\end{equation}
The evolution operator of this equation is given by
$$
\Phi(t,s)(u,v)=(U(t,s)u,V(t,s)v),
$$
where
\[
U(t, s) =\left(\frac{\mu(t)}{\mu(s)}\right)^\alpha \frac{\mu(t)^{\theta\sign(t)(\cos t-1) / 2}}{\mu(s)^{\theta\sign(s)(\cos s-1) / 2}}
\]
and
\[
V(t, s) =\left(\frac{\mu(t)}{\mu(s)}\right)^\beta \frac{\mu(t)^{\nu\sign(t)(\cos t-1) / 2}}{\mu(s)^{\nu\sign(s)(\cos s-1) / 2}}
\]
For the projections $P_1(t): \mathbb{R}^{2} \rightarrow \mathbb{R}^{2}$ defined by $P_1(t)(u, v)=(u, 0)$ we have
$$
\begin{aligned}
&\left\|\Phi(t,s) P_1(s)\right\|=|U(t, s)| \leqslant \left(\frac{\mu(t)}{\mu(s)}\right)^\alpha \mu(s)^{\sign(s) \theta}, \quad \quad \text{for} \quad t \ge s\\
&\left\|\Phi(t,s) Q_1(t)\right\|=\left|V(t, s)\right| \leqslant \left(\frac{\mu(t)}{\mu(s)}\right)^\beta \mu(s)^{\sign(s) \nu}, \quad \quad \text{for} \quad t \le s
\end{aligned}
$$
and thus the equation admits a N$\mu$D. Moreover, if $t=2 k \pi$ and $s=(2 k-1) \pi, k \in \mathbb{N}$, then
\begin{equation}\label{eq:no-unif-dic-1}
\left\|\Phi(t,s) P_1(s)\right\|=\left(\frac{\mu(t)}{\mu(s)}\right)^\alpha \mu(s)^{\sign(s)\theta}
\end{equation}
and if $t=(2 k-1) \pi$ and $s=2 k \pi, k \in \mathbb{N}$, then
\begin{equation}\label{eq:no-unif-dic-2}
\left\|\Phi(t,s) Q_1(t)\right\|=\left(\frac{\mu(t)}{\mu(s)}\right)^\beta \mu(t)^{\sign(t)\nu}.
\end{equation}
Note now that
$$
\Phi_\gamma(t,s)(u,v)=(U_\gamma(t,s)u,V_\gamma(t,s)v),
$$
where
\[
U_\gamma(t, s) =\left(\frac{\mu(t)}{\mu(s)}\right)^{-\gamma} U(t, s)
\quad\quad \text{and} \quad\quad
V_\gamma(t, s) =\left(\frac{\mu(t)}{\mu(s)}\right)^{-\gamma} V(t, s).
\]
Thus
$$
\begin{aligned}
&\left\|\Phi_\gamma(t,s) P_1(s)\right\|\leqslant \left(\frac{\mu(t)}{\mu(s)}\right)^{\alpha-\gamma} \mu(s)^{\sign(s) \theta}, \quad \quad \text{for} \quad t \ge s\\
&\left\|\Phi_\gamma(t,s) Q_1(t)\right\|\leqslant \left(\frac{\mu(t)}{\mu(s)}\right)^{\beta-\gamma} \mu(s)^{\sign(s) \nu}, \quad \quad \text{for} \quad t \le s.
\end{aligned}
$$
We conclude that if $\alpha-\gamma+\theta<0 \ \Leftrightarrow \ \gamma >\alpha+\theta$ and $\beta-\gamma-\nu>0 \ \Leftrightarrow \ \gamma < \beta-\nu$, that is if $\gamma \in \, (\alpha+\theta,\, \beta-\nu)$, then equation~\eqref{eq:exemplo} admits a N$\mu$D.

Letting $P_2(t)=I$, where $I$ the identity, for all $t \in \R$, we have
$$
\left\|\Phi_\gamma(t,s) P_2(s)\right\|\leqslant \left(\frac{\mu(t)}{\mu(s)}\right)^{\beta-\gamma} \mu(s)^{\sign(s) \nu}, \quad \quad \text{for} \quad t \ge s
$$
and we conclude that if $\beta-\gamma+\nu<0 \ \Leftrightarrow \ \gamma >\beta+\nu$, that is, if $\gamma \in \, (\beta+\nu,\, +\infty)$, then equation~\eqref{eq:exemplo} admits a N$\mu$D.

Letting $P_3(t)=0$, where $0$ the null projection, for all $t \in \R$, we have
$$
\left\|\Phi_\gamma(t,s) Q_3(t)\right\|\leqslant \left(\frac{\mu(t)}{\mu(s)}\right)^{\alpha-\gamma} \mu(s)^{\sign(s) \theta}, \quad \quad \text{for} \quad t \le s
$$
and we conclude that if $\alpha-\gamma-\theta>0 \ \Leftrightarrow \ \gamma <\alpha-\theta$, that is if $\gamma \in \, (-\infty,\alpha-\theta)$, then equation~\eqref{eq:exemplo} admits a N$\mu$D.
    By the above, we can conclude that
    $$\SigmaND_\mu(A)=[\alpha-\theta,\alpha+\theta]\cup[\beta-\nu,\beta+\nu].$$
Letting $\gamma_0 \in \, (-\infty,\alpha-\theta)$, $\gamma_1 \in \, (\alpha+\theta,\, \beta-\nu)$ and $\gamma_2 \in \, (\beta+\nu,\, +\infty)$, we have
$\mathcal U_{\gamma_0}=\mathcal V_{\gamma_2}=\R \times \{(0,0)\}$,
$\mathcal U_{\gamma_2}=\mathcal V_{\gamma_0}=\R \times \R^2$,
$\mathcal U_{\gamma_1} =\R \times \{(r,0)\in \R^2: r \in \R\}$
and $\mathcal V_{\gamma_1}=\R \times \{(0,r)\in \R^2: r \in \R\}$.
Thus
\[
\mathcal W_0=\mathcal U_{\gamma_0}=\R \times \{(0,0)\},
\]
\[
\mathcal W_1=\mathcal U_{\gamma_1} \cap \mathcal V_{\gamma_0}=\mathcal U_{\gamma_1}
=\R \times \sg \{(1,0)\},
\]
\[
\mathcal W_2=\mathcal U_{\gamma_2} \cap \mathcal V_{\gamma_1}=\mathcal V_{\gamma_1}=\R \times \sg \{(0,1)\},
\]
and
\[
\mathcal W_3=\mathcal V_{\gamma_2}=\R \times \{(0,0)\}.
\]
Finally,
\[
\R\times\R^n=\mathcal W_0\oplus\mathcal W_1\oplus\mathcal W_2\oplus\mathcal W_3=\mathcal W_1\oplus\mathcal W_2.
\]
Note that $\mathcal W_1$ is the integral manifold corresponding to forward nonuniform contraction with respect to the growth rate $\mu$ and $\mathcal W_2$ is the integral manifold corresponding to backward nonuniform contraction with respect to the growth rate $\mu$. Furthermore, by~\eqref{eq:no-unif-dic-1}--\eqref{eq:no-unif-dic-2}, there is no (uniform) $\mu$-dichotomy for any $\gamma$ and we have $\SigmaD_\mu(A)=\R$. Notice also that, if
\[
\lim_{t \to +\infty} \frac{\mu(t)}{\e^{at}}=0,
\]
for each $a>0$, then equations~\eqref{eq:no-unif-dic-1}--\eqref{eq:no-unif-dic-2} imply that $\SigmaND(A)=\R$. In particular, the nonuniform polynomial dichotomy spectrum for equation~\eqref{eq:exemplo}, with $\mu$ given by~\eqref{eq:poly-dich-R}, is different from the nonuniform dichotomy spectrum for equation~\eqref{eq:exemplo} with $\mu$ given by~\eqref{eq:poly-dich-R} (that in this case is the whole $\R$).
\end{example}

To better understand the information one can obtain from the nonuniform $\mu$-dichotomy spectrum, we consider a notion of Lyapunov exponent for nonuniform $\mu$-dichotomies. This generalized type of Lyapunov exponents were introduced by Barreira and Valls in~\cite{Barreira-Valls-NA-2009} (in the particular case of polynomial growth). In that paper, the authors show that the notion of polynomial Lyapunov exponent is in fact a Lyapunov exponent in the sense of the abstract theory~\cite{Barreira-Pesin-ULS-2002}. See also~\cite{Barreira-Valls-DCDS-2012}.

With the usual convention that $\log 0=-\infty$, define the $\mu$-Lyapunov exponent associated with system~\eqref{eq:sist-linear} as the function
$\lambda^+: \R^n \to [-\infty,+\infty]$ given by
$$
\lambda^+(v)=\limsup_{t \rightarrow +\infty} \frac{\log \|\Phi(t,s)v\|}{\log \mu(t)}.
$$

Also consider the function $\lambda^-: \R^n \to [-\infty,+\infty]$ given by
$$
\lambda^-(v)=\liminf_{t \rightarrow +\infty} \frac{\log \|\Phi(t,s)v\|}{\log \mu(t)}.
$$

The next result relates the numbers above with the nununiform $\mu$-spectrum of system~\eqref{eq:sist-linear}.
\begin{theorem}\label{thm:relation-Lyap-exp}
Let $\mu:\R\to\R^+$ be a differentiable growth rate, $m > 1$, and let $\SigmaND_\mu(A)$
be the nonuniform $\mu$-dichotomy spectrum of system \eqref{eq:sist-linear}:
$$\SigmaND_\mu(A)=I_1 \cup [a_2,b_2]\cup \cdots \cup [a_{m-1},b_{m-1}] \cup I_m,$$
where $I_1=[a_1,b_1]$ or $I_1=(-\infty,b_1]$ and $I_m=[a_m,b_m]$ or $I_m=[a_m,+\infty)$.
Given a bounded connected component $[a_i,b_i]$ of $\SigmaND_\mu(A)$, the following holds: if $(s,v) \in \mathcal W_i={\mathcal U}_{\gamma_i} \cap \mathcal V_{\gamma_{i-1}}\setminus \{0\}$, we have
$$a_i \le \lambda^-(v)\le \lambda^+(v) \le b_i,$$
for $i=1,\ldots,m$.
\end{theorem}

\begin{proof}
Let $\gamma_i \in (b_i,a_{i+1})$, $i=1,\ldots,m$.  Since $\gamma_i \notin \SigmaND_\mu(A)$, $x'=\left(A(t)-\gamma_i\frac{\mu'(t)}{\mu(t)}I\right)x$ admits a nonuniform $\mu$-dichotomy. Thus, there is a family of projections $P(t)\in M_n(\R)$, $t\in\R$, such that, for all $t,s\in\R$,
$$P(t)\Phi(t,s)=\Phi(t,s)P(s),$$
and there are constants $K\ge 1$, $\alpha<0$, $\beta>0$ and $\theta,\nu\ge 0$, with $\alpha+\theta<0$ and $\beta-\nu>0$, such that
\begin{equation*}\label{eq:dichotomy1-aa}
\|\Phi_{\gamma_i}(t,s)P(s)\| \le  K \left(\frac{\mu(t)}{\mu(s)}\right)^{\alpha}\mu(s)^{\sign(s)\theta} \mbox{ for } t\ge s,
\end{equation*}
\begin{equation*}\label{eq:dichotomy2-aa}
\|\Phi_{\gamma_i}(t,s)Q(s)\| \le  K\left(\frac{\mu(t)}{\mu(s)}\right)^{\beta}\mu(s)^{\sign(s)\nu} \mbox{ for } t\le s,
\end{equation*}
and thus
\begin{equation*}\label{eq:dichotomy1-bb}
\|\Phi(t,s)P(s)\| \le  K \left(\frac{\mu(t)}{\mu(s)}\right)^{\alpha+\gamma_i}\mu(s)^{\sign(s)\theta} \mbox{ for } t\ge s,
\end{equation*}
\begin{equation*}\label{eq:dichotomy2-bb}
\|\Phi(t,s)Q(s)\| \le  K\left(\frac{\mu(t)}{\mu(s)}\right)^{\beta+\gamma_i}\mu(s)^{\sign(s)\nu} \mbox{ for } t\le s,
\end{equation*}
By Lemma~\ref{prop:exist-dich}, we have $\mbox{\rm Im} P=\mathcal U_{\gamma_i}$ and $\mbox{\rm Ker} P=\mathcal V_{\gamma_{i-1}}$. Therefore, for each $v \in \mathcal U_i$, we have
\[
\begin{split}
\lambda^+(v)
& =\limsup_{t \rightarrow +\infty} \frac{\log \|\Phi(t,s)v\|}{\log \mu(t)}
=\limsup_{t \rightarrow +\infty} \frac{\log \|\Phi(t,s)P(s)v\|}{\log \mu(t)}\\
& \le \limsup_{t \rightarrow +\infty} \frac{\log \left[K \left(\frac{\mu(t)}{\mu(s)}\right)^{\alpha+\gamma_i}\mu(s)^{\sign(s)\theta}\right]}{\log \mu(t)}=\alpha+\gamma_i.
\end{split}
\]
Letting $\gamma_i \to b_i$, we conclude that $\lambda^+(v) \le \alpha+b_i < b_i$.
Additionally, for each $w \in \mathcal V_{i-1}$ and $s \ge t$, we have
\[
\|\Phi(s,t)w\| \ge \frac{1}{K} \left(\frac{\mu(t)}{\mu(s)}\right)^{-\gamma_{i-1}-\beta}\mu(s)^{-\sign(s)\nu}\|w\|.
\]
and thus
\[
\begin{split}
\lambda^-(v)
& =\liminf_{s \rightarrow +\infty} \frac{\log \|\Phi(s,t)w\|}{\log \mu(s)}\\
& \ge \liminf_{s \rightarrow +\infty} \frac{\log \left[\frac{1}{K} \left(\frac{\mu(t)}{\mu(s)}\right)^{-\gamma_{i-1}-\beta}\mu(s)^{-\sign(s)\nu}\|w\|\right]}{\log \mu(s)}
=\beta+\gamma_{i-1}-\nu.
\end{split}
\]
Letting $\gamma_{i-1} \to a_i$ and since $\beta-\nu>0$, we have $\lambda^-(v) \ge \beta+a_i -\nu > a_i$. The result follows.
\end{proof}

\begin{example}
In example~\ref{example1}, consider the particular case of the polynomial grow rates: take $\mu(t)$ to be the growth rate $p(t)$ in~\eqref{eq:poly-dich-R}, $\alpha=-2$, $\beta=2$ and $\theta=\nu=1$. We get the system:
\[
\small
\left[
\begin{array}{c}
u'\\ v'
\end{array}
\right]
=
A(t)
\left[
\begin{array}{c}
u\\ v
\end{array}
\right]
\quad \text{and} \quad
A(t)=
\left[
\begin{array}{cc}
a_1(t) & 0 \\
0 & a_2(t)\\
\end{array}
\right].
\]
and, for $i=1,2$,
\[
a_i(t)=\frac{(-1)^i \, 4+\sign(t)(\cos t +1)}{2(1+|t|)}-\sign(t)\log(1+|t|)\frac{\sin t}{2}.
\]
It follows from Example 9 that, for the system above, the nonuniform polynomial spectrum and the nonuniform exponential spectrum are, respectively,
$$\SigmaND_p(A)=[-3,-1]\cup[1,3] \quad \quad \text{ and } \quad \quad \SigmaND(A)=\R.$$
With respect to the spectrum $\SigmaND_p$, note that, by Theorem~\ref{thm:relation-Lyap-exp}, the interval $[-3,-1]$ indicates the existence of a linear integral manifold  where we have nonuniform polynomial contraction (since the polynomial Lyapunov exponents satisfy $-3 \le \lambda^-(v)\le \lambda^+(v) \le -1$ for $(s,v) \in \mathcal W_1$) and the interval $[1,3]$ indicates the existence of a linear integral manifold  where we have nonuniform polynomial expansion (since the polynomial Lyapunov exponents satisfy $1 \le \lambda^-(v)\le \lambda^+(v) \le 3$ for $(s,v) \in \mathcal W_2$). Recalling that the present system is a particular case of the family of systems in Example~\ref{example1}, we confirm once again that the linear integral manifold  where we have nonuniform polynomial contraction
is the linear manifold $\mathcal W_1=\R \times \sg \{(1,0)\}$ and the linear integral manifold  where we have nonuniform polynomial expansion
is the linear manifold $\mathcal W_2=\R \times \sg \{(0,1)\}$.

Once again we stress that the above information about polynomial contraction and polynomial expansion along the linear manifolds can not be obtained from the nonuniform (exponential) dichotomy spectrum for equation~\eqref{eq:exemplo}, that in this case is the whole $\R$, as already explained in Example~\ref{example1}.

Note that the computations we undertake in this example can immediately be done considering other growth rates in Example~\ref{example1}, illustrating the potential use of the nonuniform $\mu$-dichotomy spectrum to search for nonuniform contraction and expansion with respect to growth rates $\mu$ that are different from exponential ones.
\end{example}

\section{Kinematic similarity}\label{section:Normal formas}

In~\cite{Siegmund-JLMS-2002} the dichotomy spectrum is used to establish the existence of a normal forms for nonautonomous linear systems. A version of this result is obtained in~\cite{Zhang-JFA-2014,Chu-Liao-Siegmund-Xia-Zhu-ANA-2022} using the nonuniform dichotomy spetrum. The objective of this section is to obtain a version of these results for the nonuniform $\mu$-spectrum, obtaining the results mentioned above as a very particular case.

Given $\eps\ge 0$ and a growth rate $\mu$, we say that systems $x'=A(t)x$ and $y'=B(t)y$ are {\it nonuniformly $(\mu,\eps)$-kinematically similar} if there exists a
differentiable matrix function $S:\R\rightarrow {\rm GL}_n(\R)$ and a constant $M_\epsilon>0$ such that, for all $t\in\mathbb R$, we have
\begin{equation}\label{eq:kinematically-similar}
  \|S(t)\|\le M_\varepsilon \mu(t)^{\sign(t)\varepsilon} \quad \text{and} \quad
 \|S(t)^{-1}\|\le M_\varepsilon \mu(t)^{\sign(t)\varepsilon}.
\end{equation}
and the change of variables $x(t)=S(t)y(t)$ transforms $x'=A(t)x$ into $y'=B(t)y$. If $\eps=0$ the systems $x'=A(t)x$ and $y'=B(t)y$ are said {\it nonuniformly kinematically similar}, a notion considered in~\cite{Siegmund-JLMS-2002}.

Each $S:\R\rightarrow {\rm GL}_n(\R)$ satisfying \eqref{eq:kinematically-similar} for some $\eps\ge0$ is called a {\it nonuniform Lyapunov matrix function with respect to $\mu$} and the change of variables $x(t)=S(t)y(t)$ is said a {\it nonuniform Lyapunov transformation with respect to $\mu$}.

The previous definition is a generalization of the definitions of nonuniformly kinematically similar, considered in \cite{Siegmund-JLMS-2002,Zhang-JFA-2014} and corresponding to the special case $\mu(t)=\e^t$.

In the next result we obtain a characterization of the normal forms of nonautonomous linear systems using the nonuniform
$\mu$-dichotomy spectrum. This result contains Theorem 1.3 in~\cite{Zhang-JFA-2014} as the very particular case $\mu(t)=\e^t$, corresponding to the nonuniform dichotomy spectrum.

\begin{theorem}\label{thm:normal-form}
Assume that $A(t)$ is differentiable and that the evolution operator of system \eqref{eq:sist-linear} has nonuniformly bounded growth with respect to $(\mu,\eps)$. Assume also that system~\eqref{eq:sist-linear} has a nonuniform $\mu$-dichotomy with parameters $\alpha$, $\beta$, $\theta$ and $\nu$ such that
\begin{equation}\label{eq:3max-min}
3\max\{\theta,\nu\} - \min\{-\alpha-\theta,\beta-\nu\} \le 0.
\end{equation}
Let the nonuniform $\mu$-dichotomy spectrum be
$$\SigmaND_\mu(A)=[a_1,b_1]\cup \ldots \cup [a_m,b_m],$$
with $[a_i,b_i]\cap [a_j,b_j]=\emptyset$ for any $i,j \in \{1,\ldots,m\}$ such that $i\ne j$. Then system \eqref{eq:sist-linear} is $(\mu,\eps)$-nonuniformly kinematically similar to system $y'=B(t)y$, where $$B(t)=\text{diag}(B_0(t),B_1(t),\cdots,B_{m}(t),B_{m+1}(t)),$$ the matrices
$B_i(t):\,\R\rightarrow \R^{n_i\times n_i}$, with $n_i=\rank\,\mathcal W_i$, are
differentiable and
\[
\SigmaND_\mu(B_i)=
\begin{cases}
\emptyset & \quad \text{for} \quad i\in\{0,m+1\}\\
[a_i,b_i] & \quad \text{for} \quad i\in\{1,\dots,m\}
\end{cases}.
\]
\end{theorem}

\begin{proof}
To prove the result we begin by establishing some lemmas.

The proof of the next lemma is obtained using arguments borrowed from the proof of Lemma 3.1 in~\cite{Zhang-JFA-2014}, that in turn is based on the arguments in the proof of Lemma 2.1 in page 158 of Daleckii and Krein~\cite{Daleckii-Krein-livro-1974}. We present a full proof for the sake of completeness.

\begin{lemma}\label{lemma-lyap-matrix-S}
Let $S(t)$ be a nonuniform Lyapunov matrix function with respect to $\mu$. Then, the following statements are equivalent.
\begin{enumerate}[\lb=$\alph*)$,\lm=6mm,\is=2mm]
\item \label{lemma-lyap-matrix-S-1} For some $\eps\ge 0$, systems $x'=A(t)x$ and $y'=B(t)y$ are nonuniformly $(\mu,\eps)$-kinematically similar via the transformation $x=S(t)y$.
\item \label{lemma-lyap-matrix-S-2} We have $\Phi_A(t,s)S(s)=S(t)\Phi_B(t,s)$ for all $t,s\in \R$, where $\Phi_A$ and $\Phi_B$ are the evolution operators of systems
 $x'=A(t)x$ and $y'=B(t)y$, respectively.
\item \label{lemma-lyap-matrix-S-3} $S(t)$ is a solution of $S'=A(t)S-SB(t)$.
\end{enumerate}
\end{lemma}

\begin{proof}
Assume that systems $y'=A(t)y$ and $y'=B(t)y$ are nonuniformly $(\mu,\eps)$-kinematically similar via $x=S(t)y$, where $S:\R\rightarrow {\rm GL}_n(\R)$ is a differentiable matrix function such that~\eqref{eq:kinematically-similar} holds. Let $s_0\ge 0$, $y_0 \in \R^n$ and $y(t)$ be a solution of $y'=B(t)y$ with $y(s)=y_0$. Let $x(t)$ denote the solution of $x'=A(t)x$ with $x(s)=S(s)y_0$. Denoting by $\Phi_A$ and $\Phi_B$ the evolution operators of systems $x'=A(t)x$ and $y'=B(t)y$, respectively, we have
$$\Phi_A(t,s)S(s)y_0=\Phi_A(t,s)x(s)=x(t)=S(t)y(t)=S(t)\Phi_B(t,s)y_0.$$
 Since $y_0$ is arbitrary, we conclude that $\Phi_A(t,s)S(s)=S(t)\Phi_B(t,s)$, for all $s,t \in\R$. We conclude that~\ref{lemma-lyap-matrix-S-1} implies~\ref{lemma-lyap-matrix-S-2}.

Assume now that $\Phi_A(t,s)S(s)=S(t)\Phi_B(t,s)$ for all $t,s\in\R$. We have in particular $\Phi_A(t,0)S(0)\Phi_B(t,0)^{-1}=S(t)$. Since $\Phi_A(t,0)S(0)$ and $\Phi_B(t,0)$ are fundamental matrices, respectively, of systems $x'=A(t)x$ and $y'=B(t)y$, we conclude that $S(t)$ is differentiable and thus
\[
\begin{split}
S'(t) & =\frac{d}{dt} \Phi_A(t,0)S(0)\Phi_B(t,0)^{-1} \\
& = A(t)\Phi_A(t,0)S(0)\Phi_B(t,0)^{-1}-\Phi_A(t,0)S(0)\Phi_B(t,0)^{-1}B(t)\Phi_B(t,0)\Phi_B(t,0)^{-1}\\
& = A(t)\Phi_A(t,0)\Phi_A(t,0)^{-1}S(t)-S(t)\Phi_B(t,0)\Phi_B(t,0)^{-1}B(t)\\
& = A(t)S(t)-S(t)B(t)
\end{split}
\]
We conclude that~\ref{lemma-lyap-matrix-S-2} implies~\ref{lemma-lyap-matrix-S-3}.

Assume now that $S'(t)=A(t)S(t)-S(t)B(t)$ and let $S(t_0)=C$, where $C$ is an invertible matrix. Defining
 $Z(t)=\Phi_A(t,t_0)C\Phi_B(t,t_0)^{-1}$ we get
 \[
 \begin{split}
 \frac{d}{dt} Z(t)
 & = \frac{d}{dt} \left( \Phi_A(t,t_0)C\Phi_B(t,t_0)^{-1}\right)\\
 & = A(t)\Phi_A(t,t_0)C\Phi_B(t,t_0)^{-1}-\Phi_A(t,t_0)C\Phi_B(t,t_0)^{-1}B(t)\\
& = A(t)Z(t)-Z(t)B(t)
\end{split}
 \]
 and $Z(t_0)=C$. We conclude that $S(t)=Z(t)=\Phi_A(t,t_0)C\Phi_B(t,t_0)^{-1}$. Let now $x(t)=S(t)y(t)$, where $y$ is the solution of $y'=B(t)y$ with $y(t_0)=y_0$. We have
 $$x'=(S(t)y)'=S'(t)y+S(t)y'=A(t)S(t)y-S(t)B(t)y+S(t)B(t)y=A(t)x$$
 and $x(t_0)=S(t_0)y(t_0)=S(t_0)y_0$. Since $S(t)$ is a nonuniform Lyapunov matrix function with respect to $\mu$, we conclude that $y'=A(t)y$ and $y'=B(t)y$ are nonuniformly $(\mu,\eps)$-kinematically similar via a transformation $x=S(t)y$ and thus~\ref{lemma-lyap-matrix-S-3} implies~\ref{lemma-lyap-matrix-S-1}.
\end{proof}

\begin{lemma}\label{lemma:kinematically-similar-implies-dich-spectrum}
 Let $y'=A(t)y$ and $y'=B(t)y$ be nonuniformly $(\mu,\eps)$-kinematically similar and
 \begin{equation}\label{eq:3eps-min}
 \eps\le\frac13 \min\{-\alpha-\theta,\beta-\nu\}.
 \end{equation}
 Then $\SigmaND_\mu(A)=\SigmaND_\mu(B)$.
 \end{lemma}

\begin{proof}
Assume that $x'=A(t)x$ and $y'=B(t)y$ are nonuniformly $(\mu,\eps)$-kinematically similar with $\eps$ satisfying~\eqref{eq:3eps-min}. Let $S:\R\rightarrow {\rm GL}_n(\R)$ be a differentiable matrix function such that $y(t)=S(t)x(t)$ transforms $x'=A(t)x$ into $y'=B(t)y$ and satisfies $\|S(t)\|\le M_\varepsilon \mu(t)^{\sign(t)\varepsilon}$ and $\|S(t)^{-1}\|\le M_\varepsilon \mu(t)^{\sign(t)\varepsilon}$, for all $t\in\mathbb \R$ with $\eps$ satisfying~\eqref{eq:3eps-min}. Define $Y(t)=S(t)X(t)$, where $X(t)$ is a fundamental matrix of $x'=A(t)x$. Assume that $x'=A(t)x$ satisfies~\eqref{eq:dichotomy1-fund-matrix}--\eqref{eq:dichotomy2-fund-matrix}. The proof of Lemma~\ref{lemma-lyap-matrix-S} showed immediately that $Y(t)$ is a fundamental matrix of $y'=B(t)y$. Moreover, for $t \ge 0$ and $t \ge s$, we get
\[
\begin{split}
\|Y(t)\widetilde P Y(s)^{-1}\|
& \le \|S(t)\|\|X(t)\widetilde P X(s)^{-1}\|\|S(s)^{-1}\|\\
& \le M_\varepsilon \mu(t)^{\sign(t)\varepsilon} K \left(\frac{\mu(t)}{\mu(s)}\right)^\alpha \mu(s)^{\sign(s)\theta} M_\varepsilon \mu(s)^{\sign(s)\varepsilon}\\
& \le K(M_\varepsilon)^2 \left(\frac{\mu(t)}{\mu(s)}\right)^{\alpha+\varepsilon} \mu(s)^{\sign(s)(\theta+2\varepsilon)}\\
&= K(M_\varepsilon)^2 \left(\frac{\mu(t)}{\mu(s)}\right)^{\widetilde \alpha} \mu(s)^{\sign(s)\widetilde \theta},
\end{split}
\]
where $\widetilde \alpha=\alpha+\varepsilon$ and $\widetilde \theta=\theta+2\varepsilon$ (since $\mu$ is increasing and $\mu(0)=1$, we have $\mu(s)\le\mu(s)^{\sign(s)}$ for all $s \in \R$). Assuming now that $t < 0$ and $t \ge s$, we have
\[
\begin{split}
\|Y(t)\widetilde P Y(s)^{-1}\|
& \le K(M_\varepsilon)^2 \left(\frac{\mu(t)}{\mu(s)}\right)^{\alpha} \mu(s)^{\sign(s)(\theta+2\varepsilon)}\\
& = K(M_\varepsilon)^2 \left(\frac{\mu(t)}{\mu(s)}\right)^{\widetilde \alpha} \mu(s)^{\sign(s)\widetilde \theta}.
\end{split}
\]
According to~\eqref{eq:3eps-min}, $\widetilde \alpha+\widetilde \theta<0$.

Similarly, for $t<0$ and $s \ge t$ we have
\[
\begin{split}
\|Y(t)\widetilde Q Y(s)^{-1}\|
& \le \|S(t)\|\|X(t)\widetilde Q X(s)^{-1}\|\|S(s)^{-1}\|\\
& \le M_\epsilon \mu(t)^{\sign(t)\varepsilon} K \left(\frac{\mu(t)}{\mu(s)}\right)^\beta \mu(s)^{\sign(s)\nu} M_\epsilon \mu(s)^{\sign(s)\varepsilon}\\
& \le K(M_\epsilon)^2 \left(\frac{\mu(t)}{\mu(s)}\right)^{\beta-\eps} \mu(s)^{\sign(s)(\nu+2\varepsilon)}\\
& \le K(M_\epsilon)^2 \left(\frac{\mu(t)}{\mu(s)}\right)^{\widetilde \beta} \mu(s)^{\widetilde \nu},
\end{split}
\]
where $\widetilde \beta=\beta-\eps$ and $\widetilde \nu=\nu+2\epsilon$.
Also, for $t\ge 0$ and $s \ge t$ we obtain
\[
\begin{split}
\|Y(t)\widetilde Q Y(s)^{-1}\|
& \le K(M_\epsilon)^2 \left(\frac{\mu(t)}{\mu(s)}\right)^{\beta+\eps} \mu(s)^{\sign(s)(\nu+2\epsilon)}\\
& \le K(M_\epsilon)^2 \left(\frac{\mu(t)}{\mu(s)}\right)^{\widetilde \beta} \mu(s)^{\widetilde \nu}.
\end{split}
\]
By~\eqref{eq:3eps-min} we have $\widetilde \beta-\widetilde \nu>0$. We conclude that the existence of a N$\mu$D for $x'=A(t)x$ implies the existence of a N$\mu$D for $x'=B(t)x$. Reciprocally, since $H:\R\rightarrow {\rm GL}_n(\R)$ given by $H(t)=S^{-1}(t)$ is a differentiable matrix function such that $x(t)=H(t)y(t)$ transforms $y'=B(t)y$ into $x'=A(t)x$ and satisfies $\|H(t)\|\le M_\varepsilon \mu(t)^{\sign(t) \varepsilon}$ and $\|H(t)^{-1}\|\le M_\varepsilon \mu(t)^{\sign(t) \varepsilon}$, we conclude that there is a N$\mu$D for $x'=A(t)x$ if and only if there is a N$\mu$D for $x'=B(t)x$. Thus $\SigmaND_\mu(A)=\SigmaND_\mu(B)$.
\end{proof}

The proof of next lemma is contained in the proof of Lemma A.5 of \cite{Siegmund-JLMS-2002} and was already used in the present form in~\cite{Zhang-JFA-2014}.

\begin{lemma}[Lemma~3.3 of~\cite{Zhang-JFA-2014}]\label{lemma-properties-projection}
Let $P_0\in M_n(\R)$ be a symmetric projection and
$\X:\R \to {\rm GL}_n(\R)$.Then:
\begin{enumerate}[\lb=$\alph*)$,\lm=6mm,\is=2mm]
\item \label{lemma-properties-projection-1} For every $t \in\R$, the matrices $Q(t) \in M_n(\R)$ given by
\begin{equation}\label{eq:Q(t)-normal-form}
Q(t)=P_0X(t)^TX(t)P_0+(I-P_0)X(t)^TX(t)(I-P_0)
\end{equation}
are positively definite and symmetric. Moreover, there exists a unique positively definite and
symmetric matrix $R(t)$ such that $R(t)^2=Q(t)$ and
$P_0R(t)=R(t)P_0$ for every $t \in \R$;
\item \label{lemma-properties-projection-2} The matrix $S(t)=X(t)R(t)^{-1}$ is invertible, $S(t)P_0S(t)^{-1}=X(t)P_0X(t)^{-1}$ and we have, for every $t \in \R $,
\[
\|S(t)\|\le \sqrt{2},\quad
\|S(t)^{-1}\|\le\sqrt{\|X(t)P_0X(t)^{-1}\|^2+\|X(t)(I-P_0)X(t)^{-1}\|^2}.
\]
\end{enumerate}
\end{lemma}

\begin{lemma}\label{lemma:exist-nonunif-Lyap-function-wrt-mu}
Assume that system \eqref{eq:sist-linear} has a N$\mu$D with an invariant projection
$P:\,\R\rightarrow M_n(\R)$ satisfying $P(t)\ne 0,I$. Then there
exists a differentiable nonuniform Lyapunov matrix function with respect to $\mu$, $S: \, \R\rightarrow{\rm
GL}_n(\R)$, such that, for all $t\in \R$, we have $\|S(t)\|, \, \|S(t)^{-1}\| \le K \mu(t)^{\sign(t)\max\{\theta,\nu\}}$, where $K$, $\theta$ and $\nu$ are the constants in~\eqref{eq:dichotomy1}--\eqref{eq:dichotomy2}, and
\[
S(t)^{-1}P(t)S(t)=
\left[\begin{array}{cc} I & 0\\ 0 &
0\end{array}\right].
\]
\end{lemma}

\begin{proof}
All the projections $P(t)$ have the same rank since $P(t)\Phi(t,s)=\Phi(t,s)P(s)$, for $t,s\in \R$,
and this property implies that, for all $t,s\in \R$, the matrices $P(t)$ and $P(s)$ are similar.
By Remark~\ref{remark:form-of-projections}, for
any given $s\in \R$ there exists a $T(s)\in{GL}_n(\R)$ such that, for all $s\in \R$,
\begin{equation}\label{eq:hipotese-TPT-1=P0}
T(s)P(s)T(s)^{-1}=
\left[\begin{array}{cc} I & 0\\ 0 &
0\end{array}\right]=:P_0,
\end{equation}
where $I$ is the identity of dimension $\dim{\rm Im}P$. Let $X(t)=\Phi(t,s)T(s)^{-1}$.
By Lemma~\ref{lemma-properties-projection},  for all $t\in \R$, there is a unique positively definite and
symmetric matrix $R(t)$ satisfying $P_0R(t)=R(t)P_0$ and $R(t)^2=Q(t)$ with $Q(t)$ given by~\eqref{eq:Q(t)-normal-form}. Letting
$S(t)=\Phi(t,s)T(s)^{-1}R(t)^{-1}$, using the
invariance of $P(t)$ and~\eqref{eq:hipotese-TPT-1=P0} we have
\[
\begin{split}
S(t)^{-1}P(t)S(t)
& = R(t)T(s)\Phi(t,s)^{-1}P(t)\Phi(t,s)T(s)^{-1}R(t)^{-1}\\
& =R(t)T(s)P(s)T(s)^{-1}R(t)^{-1}\\
& =R(t)P_0R(t)^{-1}=P_0.
\end{split}
\]
Finally, by Lemma~\ref{lemma-properties-projection}, we have
$\|S(t)\|\le \sqrt{2}$ and, using the fact that system \eqref{eq:sist-linear} has a N$\mu$D with an invariant projection
$P:\,\R\rightarrow M_n(\R)$ satisfying $P(t)\ne 0,I$ and estimates~\eqref{eq:dichotomy1-fund-matrix}--\eqref{eq:dichotomy2-fund-matrix}, we conclude that
\[
\begin{split}
\|S(t)^{-1}\|
& \le \sqrt{\|X(t)P_0X(t)^{-1}\|^2+\|X(t)(I-P_0)X(t)^{-1}\|^2}\\
& \le \sqrt{K^2 \mu(t)^{2\sign(t)\theta}+K^2\mu(t)^{2\sign(t)\nu}} \le \sqrt{2}K \mu(t)^{\sign(t)\max\{\theta,\nu\}}.
\end{split}
\]
Finally, note that the differentiability of $S(t)$ follows from the following facts: $R(t)$ is the positive square root of a differentiable, positive definite and symmetric matrix valued function; positive square roots of differentiable, positive definite and symmetric matrix valued functions are differentiable and the inverse of an invertible differentiable matrix valued function is still differentiable (see Coppel~\cite{Coppel-LNM-1978}, Lemma 1, page 39).

Thus $S(t)$ is a nonuniform Lyapunov matrix function with respect to $\mu$ and we establish the lemma.
\end{proof}

Next we will use the auxiliary lemmas to establish our result. According to Theorem~\ref{thm:spectrum}, we have $\mathcal W_0\oplus\mathcal W_1\oplus\ldots\oplus\mathcal W_m\oplus \mathcal W_{m+1}=\R\times \R^n$ with $\rank\mathcal W_i\ge 1$ for
$i=1,\ldots,m$.

Choose $\gamma_0\in\,(-\infty,a_1)$, $\gamma_m\in\,(b_m,+\infty)$ and
$\gamma_i\in\,(b_i,a_{i+1})$ for $i=1,\ldots,m-1$. According to Theorems~\ref{thm:spectrum} and~\ref{thm:bounded} we have $\mathcal W_0=\mathcal U_{\gamma_0}$, $\mathcal W_{m+1}=\mathcal V_{\gamma_m}$ and $\mathcal W_i=\mathcal U_{\gamma_i} \cap V_{\gamma_{i-1}}$ for $i=1,\ldots,m$.

For $\gamma_0 \in (-\infty,a_1) \subset \rhoND_\mu(A)$, the system
\begin{equation}\label{eq:At-gamma0}
x'=\left(A(t)-\gamma \frac{\mu'(t)}{\mu(t)} I\right) x,
\end{equation}
with $\gamma=\gamma_0$, admits a N$\mu$D with some invariant projection $\widetilde P_0$: if $\Phi_{\gamma_0}(t,s)$ denotes the evolution operator of~\eqref{eq:At-gamma0} with $\gamma=\gamma_0$, then there are $K_0\ge 1$, $\alpha_0<0<\beta_0$ and $\theta_0,\nu_0\ge 0$ satisfying $\alpha_0+\theta_0<0$ and $\beta_0-\nu_0>0$ such that~\eqref{eq:dichotomy1}--\eqref{eq:dichotomy2} hold with $\Phi=\Phi_{\gamma_0}$. The argument used to obtain~\eqref{eq:Phi-gamma} shows that the evolution operator of~\eqref{eq:At-gamma0}, $\Phi_\gamma(t,s)$, is given by $\Phi_{\gamma}(t,s)=\left(\frac{\mu(t)}{\mu(s)}\right)^{\gamma_0-\gamma}\Phi_{\gamma_0}(t,s)$.
Thus, for $t \ge s$ we have
\begin{equation}\label{eq:dichotomy1-b}
\begin{split}
\|\Phi_\gamma(t,s)\widetilde P_0(s)\|
& =\left(\frac{\mu(t)}{\mu(s)}\right)^{\gamma_0-\gamma}\left\|\Phi_{\gamma_0}(t,s)\widetilde P_0(s)\right\| \\
& \le  K_0 \left(\frac{\mu(t)}{\mu(s)}\right)^{\alpha_0+\gamma_0-\gamma}\mu(s)^{\sign(s)\theta_0}
\end{split}
\end{equation}
and, similarly, for $s \ge t \ge 0$ we have
\begin{equation}\label{eq:dichotomy2-b}
\begin{split}
\|\Phi_\gamma(t,s)\widetilde Q_0(s)\|
& =\left(\frac{\mu(t)}{\mu(s)}\right)^{\gamma_0-\gamma}\left\|\Phi_{\gamma_0}(t,s)\widetilde  Q_0(s)\right\| \\
& \le  K_0\left(\frac{\mu(t)}{\mu(s)}\right)^{\beta_0+\gamma_0-\gamma}\mu(s)^{\sign(s)\nu_0},
\end{split}
\end{equation}
where $\widetilde  Q_0(s)=I-\widetilde  P_0(s)$. Additionally,
\[
\begin{split}
\widetilde P_0(t)\Phi_\gamma(t,s)
& =\left(\frac{\mu(t)}{\mu(s)}\right)^{\gamma_0-\gamma}\widetilde P_0(t)\Phi_{\gamma_0}(t,s)
=\left(\frac{\mu(t)}{\mu(s)}\right)^{\gamma_0-\gamma}\Phi_{\gamma_0}(t,s)\widetilde P_0(s)\\
& =\Phi_\gamma(t,s)\widetilde P_0(s)
\end{split}
\]
and $\widetilde P_0$~is an invariant projection for $\Phi_\gamma$. In particular, $\widetilde P_0$~is an invariant projection for $\Phi$.
By Lemma~\ref{lemma:exist-nonunif-Lyap-function-wrt-mu}, there
exists a differentiable nonuniform Lyapunov matrix function with respect to $\mu$, $S_0:\,\R\rightarrow{\rm GL}_n(\R)$,
such that $\|S(t)\|,\|S(t)^{-1}\|\le M\mu(t)^{\sign(t)\max\{ | \theta |, | \nu |\}}$ and
\[
S_0(t)^{-1}\widetilde P_0(t)S_0(t)=
\left[\begin{array}{cc}
I & 0\\
0 & 0
\end{array}\right]:= P_0,
\]
where $I$ denotes the identity of order $\dim \imagem \widetilde P_0(t)$. For $t\in\R$, define
\[
B(t):=S_0(t)^{-1}(A(t)S_0(t)-S_0'(t))  \quad \Leftrightarrow \quad S_0'(t)=A(t)S_0(t)-S_0(t)B(t).
\]
By~\ref{lemma-lyap-matrix-S-3} in Lemma~\ref{lemma-lyap-matrix-S} and the identity above we conclude that system \eqref{eq:sist-linear} is nonuniformly {$(\mu,\eps_0)$-kinematically} similar to
system $y'=B(t)y$ via the transformation $x(t)=S_0(t)y(t)$, where $\eps_0=\max\{\theta,\nu\}$. Moreover, according to~\ref{lemma-lyap-matrix-S-2} in Lemma~\ref{lemma-lyap-matrix-S}, $\Phi_B(t,s)=S_0(t)^{-1}\Phi(t,s)S_0(s)$ is the evolution operator of $y'=B(t)y$.

Define $R(t)=S_0(t)^{-1}\Phi(t,s)T^{-1}(s)$, where $T(s)\in GL_n(\R)$ was defined in the proof of Lemma~\ref{lemma:exist-nonunif-Lyap-function-wrt-mu} (see~\eqref{eq:hipotese-TPT-1=P0}).
Since $P_0R(t)=R(t)P_0$, we have
$$R(t)^{-1}P_0
=P_0R(t)^{-1} \quad \text{ and } \quad R'(t)P_0=P_0R'(t).$$

 Using the
fact that $S_0(t)^{-1}\Phi(t,s)S_0(s)$ is a fundamental matrix
solution of $y'=B(t)y$, we have $R(t)T(s)S_0(s)=\Phi_B(t,s)$ and thus
$$R'(t)T(s)S_0(s)=B(t)\Phi_B(t,s)=B(t)S_0(t)^{-1}\Phi(t,s)S_0(s).$$ Therefore
$$R'(t)R(t)^{-1}=B(t)S_0(t)^{-1}\Phi(t,s)\Phi(t,s)^{-1}S_0(t)=B(t)$$
and also, since $P_0$ commutes with $R'(t)$ and $R(t)^{-1}$,
\begin{equation}\label{eq:P0Bt=BTP0}
P_0B(t)=P_0R'(t)R(t)^{-1}=R'(t)R(t)^{-1}P_0=B(t)P_0.
\end{equation}
For $t \in \R$ write
\[
B(t)
=\left[
\begin{array}{cc}
B_0(t) & C_0(t)\\
D_0(t) & E_0(t)
\end{array}
\right],
\]
where $B_0:\R\rightarrow M_{n_0}(\R)$,
$D_0:\R\rightarrow M_{n-n_0}(\R)$,
$C_0:\R\rightarrow \R^{n_0\times (n-n_0)}$,
$E_0:\R\rightarrow \R^{(n-n_0)\times n_0}$ and $n_0=\dim{\rm Im}P_0$ (and thus $\dim{\rm Ker}P_0=n-n_0$).
The identity \eqref{eq:P0Bt=BTP0} implies that $C_0(t)=D_0(t)=0$ for all $t \in \R$.
We have established that system \eqref{eq:sist-linear} is nonuniformly $(\mu,\eps_0)$-kinematically similar to $y'=D_0 y$, where
\begin{equation}\label{eq:reduced-0}
D_0
=\left[
\begin{array}{cc}
B_0(t) & 0\\
0 & E_0(t)
\end{array}
\right],
\end{equation}
$B_0:\R\rightarrow M_{n_0}(\R)$ and $E_0:\R\rightarrow M_{n-n_0}(\R)$ are differentiable, where $n_0=\dim{\rm Im}\widetilde P_0$ and $n_1=\dim{\rm Ker}\widetilde P_0$.
Taking into account~\eqref{eq:3max-min}, we have $\eps_0\le \frac13\max\{-\alpha-\theta,\beta-\nu\}$ and it follows from Lemma~\ref{lemma:kinematically-similar-implies-dich-spectrum} that systems \eqref{eq:sist-linear} and \eqref{eq:reduced-0} have the
same nonuniform $\mu$-dichotomy spectrum. Moreover, $P_0$ is an invariant projection for system \eqref{eq:reduced-0}. So we get from \eqref{eq:dichotomy1-b}--\eqref{eq:dichotomy2-b} that $\SigmaND_\mu(B_0)\subset (-\infty, a_1)$ and $\SigmaND_\mu(E_0)=\SigmaND_\mu(A)$. This implies that $\SigmaND_\mu(B_0)=\emptyset$. Note that, for each $\lambda>a_1$, equation~\eqref{eq:dichotomy1-b} implies that
$$z'=\left(B_0(t)-\lambda \frac{\mu(t)}{\mu(s)}I_{n_0}\right)z$$
admits a N$\mu$D with projection $P=I_{n_0}$.

For $\gamma_1 \in \ (b_1,a_2) \ \subset \rhoND_\mu(E_0)=\rhoND_\mu(A)$, system~\eqref{eq:At-gamma0}
with $\gamma=\gamma_1$ admits a N$\mu$D with some invariant projection $\widetilde P_1$: if $\Phi_{\gamma_1}(t,s)$ denotes the evolution operator of~\eqref{eq:At-gamma0} with $\gamma=\gamma_1$, then there are $K_1\ge 1$, $\alpha_1<0<\beta_1$ and $\theta_1,\nu_1\ge 0$ satisfying $\alpha_1+\theta_1<0$ and $\beta_1-\nu_1>0$ such that~\eqref{eq:dichotomy1}--\eqref{eq:dichotomy2} hold with $\Phi=\Phi_{\gamma_1}$.

Again, for $t \ge s$ we have
\begin{equation}\label{eq:dichotomy1-c}
\begin{split}
\|\Phi_\gamma(t,s)\widetilde P_1(s)\|
& =\left(\frac{\mu(t)}{\mu(s)}\right)^{\gamma_1-\gamma}\left\|\Phi_{\gamma_1}(t,s)\widetilde P(s)\right\| \\
& \le  K_0 \left(\frac{\mu(t)}{\mu(s)}\right)^{\alpha_1+\gamma_1-\gamma}\mu(s)^{\sign(s)\theta_1}
\end{split}
\end{equation}
and, similarly, for $s \ge t$ we have
\begin{equation}\label{eq:dichotomy2-c}
\begin{split}
\|\Phi_\gamma(t,s)\widetilde Q_1(s)\|
& =\left(\frac{\mu(t)}{\mu(s)}\right)^{\gamma_1-\gamma}\left\|\Phi_{\gamma_1}(t,s)\widetilde  Q(s)\right\|\\
& \le  K_1\left(\frac{\mu(t)}{\mu(s)}\right)^{\beta_1+\gamma_1-\gamma}\mu(s)^{\sign(s)\nu_1},
\end{split}
\end{equation}
where $\widetilde  Q_1(t)=I-\widetilde  P_1(t)$ and $\widetilde P_1$~is an invariant projection for $\Phi_\gamma$.

Reproducing the argument used to obtain~\eqref{eq:reduced-0}, we conclude that system $z'=E_0 z$ is nonuniformly $(\mu,\eps)$-kinematically similar to system
\begin{equation}\label{eq:reduced-1}
y'
=\left[
\begin{array}{cc}
B_1(t) & 0\\
0 & E_1(t)
\end{array}
\right]
y,
\end{equation}
with $B_1:\R\rightarrow M_{n_1}(\R)$ and $E_0:\R\rightarrow M_{n-n_1-n_0}(\R)$ differentiable and $n_1=\dim \widetilde P_1$. Thus, by~\eqref{eq:3max-min}, system~\eqref{eq:sist-linear} is nonuniformly $(\mu,\eps_0)$-kinematically similar to $y'=D_1y$ where
\begin{equation}\label{eq:reduced-1}
D_1=\left[
\begin{array}{ccc}
B_0(t) & 0 & 0\\
0 & B_1(t) & 0\\
0 & 0 & E_1(t)
\end{array}
\right]
\end{equation}
and $B_0:\R\rightarrow M_{n_0}(\R)$, $B_1:\R\rightarrow M_{n_1}(\R)$ and $E_1:\R\rightarrow M_{p_1}(\R)$, with $p_1=n-n_0-n_1$, are differentiable. From \eqref{eq:dichotomy1-c}, we conclude that $\SigmaND_\mu(B_1)\subset \,(-\infty, a_2)\,$ (note that when $\gamma \ge a_2$ we have $\gamma_1-\gamma \le 0$) and from \eqref{eq:dichotomy2-c}, we have $\SigmaND_\mu(E_1)\subset \,(b_1,+\infty)\,$ (note that when $\gamma \le b_1$ we have $\gamma_1-\gamma \ge 0$). Thus $(-\infty,b_1)\subset \rhoND_\mu(E_1)$ and we conclude that
$$\SigmaND_\mu(B_1)=[a_1,b_1] \quad \text{and} \quad \SigmaND_\mu(E_1)=[a_2,b_2]\cup \ldots \cup [a_m,b_m].$$

Iterating the process, we conclude that system~\eqref{eq:sist-linear} is nonuniformly $(\mu,\eps_0)$-kinematically similar to $y'=D_{m-1}y$ where
$$D_{m-1}=\text{diag}\left(B_0(t),\ldots,B_{m-1}(t),E_{m-1}(t)\right)$$
and $B_0:\R\rightarrow M_{n_0}(\R), \ldots, B_{m-1}:\R\rightarrow M_{n_{m-1}}(\R)$
and $E_{m-1}:\R\rightarrow M_{p_{m-1}}(\R)$, with $p_{m-1}=n-\sum_{k=0}^{m-1} n_k$, are differentiable.

Proceeding like before and taking into account~\eqref{eq:3max-min}, we conclude that $\SigmaND_\mu(B_i) = [a_i,b_i]$, $i=1,\ldots m-1$,
and $\SigmaND_\mu(E_{m-1})=[a_m,b_m]$.

Finally, system $z'=E_{m-1} z$ is nonuniformly $(\mu,\eps_0)$-kinematically similar to system
\begin{equation}\label{eq:reduced-1}
y'
=\left[
\begin{array}{cc}
B_m(t) & 0\\
0 & B_{m+1}(t)
\end{array}
\right]
y,
\end{equation}
with $B_m:\R\rightarrow M_{n_m}(\R)$ and $B_{m+1}:\R\rightarrow M_{n_{m+1}}(\R)$, $n_{m+1}=n-\sum_{k=0}^{m} n_k$, differentiable. Thus, system~\eqref{eq:sist-linear} is nonuniformly $(\mu,\eps_0)$-kinematically similar to $y'=D_my$ where
$$D_m=\text{diag}\left(B_0(t),\ldots,B_{m+1}(t)\right)$$
and $B_i:\R\rightarrow M_{n_i}(\R)$, $i=0,\ldots,m+1$, are differentiable.

For $\gamma_m \in \, (b_m,+\infty) \, \subset \rhoND_\mu(E_{m+1}) \subset \rhoND_\mu(A)$, system~\eqref{eq:At-gamma0}
with $\gamma=\gamma_m$, admits a N$\mu$D with some invariant projection $\widetilde P_m$: if $\Phi_{\gamma_m}(t,s)$ denotes the evolution operator of~\eqref{eq:At-gamma0} with $\gamma=\gamma_m$, then there are $K_m\ge 1$, $\alpha_m<0<\beta_m$ and $\theta_m,\nu_m\ge 0$ satisfying $\alpha_m+\theta_m<0$ and $\beta_m-\nu_m>0$ such that~\eqref{eq:dichotomy1}--\eqref{eq:dichotomy2} hold with $\Phi=\Phi_{\gamma_m}$. We conclude, using~\eqref{eq:3max-min}, that
$$\SigmaND_\mu(B_m)=[a_m,b_m] \quad \text{and} \quad \SigmaND_\mu(B_{m+1})=\emptyset.$$

It remains to prove that $n_i=\rank \mathcal W_i$. By Lemma~\ref{prop:exist-dich}, taking into account that
$\mathcal W_0=\mathcal U_{\gamma_0}$ for $\gamma_0\in \,(-\infty,b_1)\,$, we have $\ker \mbox{Im} \widetilde P_0=\ker \mathcal U_{\gamma_0}=\ker \mathcal W_0$. Thus $n_0=\dim \mathcal W_0$.

Since $\gamma_0\in \,(-\infty,a_1)\,$, $\gamma_1\in \,(b_1,a_2)\,$, $\mathcal U_{\gamma_0}\subset \mathcal U_{\gamma_1}$ and $\mathcal W_1=\mathcal U_{\gamma_1}\cap\mathcal V_{\gamma_0}$, we get from Lemmas~\ref{prop:exist-dich} and~\ref{prop:equiv} that
\[
n_0+n_1=\rank \mbox{Im} \widetilde P_1=\rank \mathcal U_{\gamma_1}=\rank(\mathcal U_{\gamma_1}\cap(\mathcal U_{\gamma_0}\oplus\mathcal V_{\gamma_0}))=\rank\mathcal W_0+\dim\mathcal W_1
\]
and we get $n_1=\rank\mathcal W_1$. Repeating the argument we immediately conclude that $n_i=\rank\mathcal W_i$ for $i=2,\ldots,m$. Using again Lemma~\ref{prop:exist-dich}, for $\gamma_m\in \,(b_m,\infty)\,$ we get
$n_{m+1}=\rank \mbox{Ker} \widetilde P_m=\rank \mathcal V_{\gamma_m}= \rank\mathcal W_{m+1}$ and the theorem follows.
\end{proof}

In the next remark we show that there is a conjugacy between the flow of an autonomous linear system associated to system $x'=A(t)x$ in Theorem~\ref{thm:normal-form}, to the flow of an autonomous linear system associated to system $x'=B(t)x$ in that result.

\begin{remark}
Assume that the linear nonautonomous system $x'=A(t)x$ satisfy the conditions of Theorem~\ref{thm:normal-form}, and that $x'=B(t)x$
is the nonuniformly $(\mu,\eps)$-kinematically similar linear block diagonal system given by the referred theorem. Denote, respectively, the evolution operators
of the systems by $\Phi_A$ and $\Phi_B$. Consider also the associated autonomous systems
\[
\begin{cases}
x'=A(t)x\\
t'=1
\end{cases}
\quad \quad \text{ and } \quad \quad
\begin{cases}
x'=B(t)x\\
t'=1
\end{cases}
\]
and the corresponding flows, $\Psi^A$ and $\Psi^B$, given by
\[
\Psi^A(\tau,(s,x))=(s+\tau,\Phi_A(s+\tau,s)x)=:\Psi_\tau^A(s,x)
\]
and
\[
\Psi^B(\tau,(s,x))=(s+\tau,\Phi_B(s+\tau,s)x)=:\Psi_\tau^B(s,x).
\]
Define a linear map $H:\R\times \R^n\to \R\times \R^n$ by $H(s,x)=(s,S(s)x)$, where $S(t)$ is given in Lemma~\ref{lemma:exist-nonunif-Lyap-function-wrt-mu}.
It follows from~\ref{lemma-lyap-matrix-S-2} in Lemma~\ref{lemma-lyap-matrix-S} that we have
\[
\begin{split}
(H\circ \Psi_\tau^B)(s,x)
& = H(s+\tau,\Phi_B(s+\tau,s)x)
= (s+\tau,S(s+\tau)\Phi_B(s+\tau,s)x)\\
& = (s+\tau,\Phi_A(s+\tau,s)S(s)x) = \Psi_\tau^A(s,S(s)x) =(\Psi_\tau^A \circ H)(s,x),
\end{split}
\]
for all $\tau,s,t \in \R$ and $x \in \R^n$. We conclude that
$$H\circ \Psi_\tau^B=\Psi_\tau^A \circ H$$
and the flows $\Psi^A$ and $\Psi^B$ are conjugated by $H$.
\end{remark}

\section{Spectrum and kinematic similarity on the half-line}\label{section:half_line}

The purpose of this section is twofold: to remark that, with the natural adaptations, our results still hold on the half-line and to present a nontrivial example. We begin by enumerate briefly the adapted versions of the definitions that we will need.

Still assuming that $t \mapsto A(t)$ is continuous, consider now that system~\eqref{eq:sist-linear} is only defined for $t \in \R_0^+$.

In the present context, a {\it growth rate} is a function $\mu:\R_0^+\to \R^+$ that is strictly increasing and satisfies $\mu(0)=1$ and $\displaystyle \lim_{t \to +\infty} \mu(t)=+\infty$.

We say that system
\begin{equation}\label{eq:sist-half-line}
x'=A(t)x, \ \ \ t\in \R_0^+
\end{equation}
admits a {\it nonuniform $\mu$-dichotomy on the half-line} (N$\mu$D$+$)
if there is a family of projections $P(t)\in M_n(\R)$, $t\in\R_0^+$, such that, for all $t,s\in\R_0^+$,
$$P(t)\Phi(t,s)=\Phi(t,s)P(s),$$
and there are constants $K\ge 1$, $\alpha<0$, $\beta>0$ and $\theta,\nu\ge 0$, with $\alpha+\theta<0$ and
$\beta-\nu>0$, such that
\begin{equation}\label{eq:dichotomy1-hl}
\|\Phi(t,s)P(s)\| \le  K \left(\frac{\mu(t)}{\mu(s)}\right)^{\alpha}\mu(s)^\theta \mbox{ for } t\ge s\ge 0,
\end{equation}
\begin{equation}\label{eq:dichotomy2-hl}
\|\Phi(t,s)Q(s)\| \le  K\left(\frac{\mu(t)}{\mu(s)}\right)^{\beta}\mu(s)^\nu \mbox{ for } 0\le t\le s,
\end{equation}
where $Q(s)=I-P(s)$ is the complementary projection. When $\theta=\nu=0$ we say that system \eqref{eq:sist-linear}
admits a {\it (uniform) $\mu$-dichotomy on the half-line ($\mu$D$+$)}.

Given a differentiable growth rate $\mu:\R_0^+\to\R^+$. We define the {\it nonuniform $\mu$-dichotomy spectrum on the half-line} by
\[
\SigmaNDp_\mu(A)=\left\{\gamma\in\R;\, x'=\left(A(t)-\gamma \frac{\mu'(t)}{\mu(t)}I\right)x \ \text{\small admits no N$\mu$D$+$ }\right\}
\]
and the {\it $\mu$-dichotomy spectrum on the half-line} by
\[
\SigmaDp_\mu(A)=\left\{\gamma\in\R;\, x'=\left(A(t)-\gamma \frac{\mu'(t)}{\mu(t)}I\right)x \ \text{ admits no $\mu$D$+$ }\right\}.
\]
Naturally, we have $\SigmaNDp_\mu(A)\subseteq \SigmaDp_\mu(A)$. We use the expressions \emph{nonuniform dichotomy spectrum on the half line} and \emph{dichotomy spectrum on the half line} to refer to the concepts obtained when the growth rate is given by $\mu(t)=\e^t$ and the expressions \emph{nonuniform polynomial spectrum on the half line} and \emph{polynomial spectrum on the half line} to refer to the concepts obtained when the growth rate is given by $p(t)=1+t$. We use the notation $\SigmaNDp(A)$ and $\SigmaNDp_{p}(A)$, respectively, for the nonuniform dichotomy spectrum on the half line and the nonuniform polynomial dichotomy spectrum on the half line.

Given $\eps\ge 0$ and a growth rate $\mu$, we still say that systems $x'=A(t)x$ and $y'=B(t)y$ are {\it nonuniformly $(\mu,\eps)$-kinematically similar} if, for $t \in \R_0^+$, they satisfy the definition of nonuniformly $(\mu,\eps)$-kinematic similarity given in section~\ref{section:Normal formas}. With the same immediate adaptation we can bring the concepts of nonuniform Lyapunov matrix function with respect to $\mu$ and nonuniform Lyapunov transformation with respect to $\mu$ to the half-line setting. We say that the evolution operator $\Phi(t,s)$ of $x'=A(t)x$ has {\it nonuniformly bounded growth with respect to $(\mu,\eps)$ on the half-line} if there are $K \ge 1$ and $a\ge 0$ such that
\begin{equation}\label{eq:bounded-growth-hl}
\|\Phi(t,s)\|\le K\left(\frac{\mu(t)}{\mu(s)}\right)^{\text{sgn}(t-s) a} \mu(s)^{\eps},\quad t,s\in\R_0^+.
\end{equation}
When $\eps=0$ the evolution operator is said to have {\it bounded growth with respect to $\mu$ on the half-line}.

The next results are versions of Theorems~\ref{thm:spectrum} and~\ref{thm:normal-form} for linear equations defined on the half-line. Their proof can be obtained by simply rewriting the proofs of those theorems in the present setting and thus we omit them.

Given $\gamma\in\R$, define the sets
\begin{equation}\label{eq:stable-inv-manif-hl}
\mathcal U_\gamma^+=\left\{(s,\xi)\in\R^+_0\times\R^n:\,\sup\limits_{t\ge
0}\|\Phi(t,s)\xi\|\mu(t)^{-\gamma}<\infty\right\}
\end{equation}
and
\begin{equation}\label{eq:unstable-inv-manif-hl}
\mathcal V_\gamma^+=\left\{(s,\xi)\in\R_0^+\times\R^n:\,\sup\limits_{t\le 0}\|\Phi(t,s)\xi\|\mu(t)^{-\gamma}<\infty\right\}.
\end{equation}

\begin{theorem}\label{thm:spectrum-hl}
Let $\mu:\R_0^+\to\R^+$ be a differentiable growth rate. The following statements hold for system \eqref{eq:sist-linear} on the half-line:
\begin{enumerate}[\lb=$\arabic*)$,\lm=5mm]
\item \label{thm:spectrum-1} There is an $m \in \{0,\ldots,n\}$ such that the nonuniform $\mu$-dichotomy spectrum on the half-line $\SigmaNDp_\mu(A)$ is the union of $m$ disjoint closed intervals
in $\R$:
\begin{enumerate}[\lb=$\alph*)$,\lm=6mm,\is=2mm]
  \item if $m=0$ then $\SigmaNDp_\mu(A)=\emptyset$;
  \item if $m=1$ then $$\SigmaNDp_\mu(A)=\mathbb R \ \text{or} \ \SigmaNDp_\mu(A)=(-\infty,b_1] \ \text{or}$$ $$\SigmaNDp_\mu(A)=[a_1,b_1] \ \text{or} \ \SigmaNDp_\mu(A)=[a_1,\infty);$$
  \item \label{thm:spectrum-1c} if $1<m\le n$ then
  $$\SigmaNDp_\mu(A)=I_1\cup [a_2,b_2]\cup\ldots\cup [a_{m-1},b_{m-1}]\cup I_m$$
  with $I_1=[a_1,b_1]$ or $(-\infty,b_1]$, $I_m=[a_m,b_m]$ or $[a_m,\infty)$ and $a_i\le b_i<a_{i+1}$ for $i=1,\ldots,m-1$.
\end{enumerate}
\item \label{thm:spectrum-2} Assume $m \ge 1$, write
    $$\SigmaNDp_\mu(A)=I_1\cup [a_2,b_2]\cup\ldots\cup [a_{m-1},b_{m-1}]\cup I_m$$
and, for $i=0,\dots,m+1$, define
\[
\mathcal W_i^+=
\begin{cases}
\mathbb R^+_0\times\{0\}  &\quad \text{ if } \ i=0 \ \text{ and } \ I_1=(-\infty,b_1]\\
\mathcal U_{\gamma_0}^+ \ \text{ for some } \gamma_0\in(-\infty,a_1) &\quad \text{ if } \ i=0 \ \text{ and } \ I_1=[a_1,b_1]\\
\mathcal U_{\gamma_i}^+ \cap \mathcal V_{\gamma_{i-1}}^+  \ \text{ for some } \gamma_i \in(b_i,a_{i+1}) &\quad \text{ if } \ i=1,\ldots,m\\
\mathcal V_{\gamma_m}^+ \ \text{ for some } \gamma_m \in(b_m,+\infty) &\quad \text{ if } \ i=m+1 \ \text{ and } \ I_m=[a_m,b_m]\\
\mathbb R_0^+ \times\{0\} \ &\quad \text{ if } \ i=m+1 \ \text{ and } \ I_m=[a_m,+\infty)\\
\end{cases}.
\]
Then, the sets $\mathcal U_{\gamma_i}^+$ $\mathcal V_{\gamma_i}^+$ and $\mathcal W_i^+$, $i=0,\ldots,m+1$, are integral manifolds, $\rank\mathcal W_i^+\ge 1$ for $i=1,\dots,m$ and
\[
\mathcal W_0^+\oplus\mathcal W_1^+\oplus\ldots\oplus\mathcal W_{m+1}^+=\R_0^+\times\R^n.
\]
\end{enumerate}
\end{theorem}

\begin{theorem}\label{thm:normal-form-hl}
Assume that $A(t)$ is differentiable, that the evolution operator of system \eqref{eq:sist-linear} on the half-line has nonuniform bounded growth with respect to $(\mu,\eps)$, and
that system~\eqref{eq:sist-half-line} has a nonuniform $\mu$-dichotomy with parameters $\alpha$, $\beta$, $\theta$ and $\nu$ such that~\eqref{eq:3max-min} holds.
Let the nonuniform $\mu$-dichotomy spectrum on the half-line be
$$\SigmaNDp_\mu(A)=[a_1,b_1]\cup \ldots \cup [a_m,b_m],$$
with $[a_i,b_i]\cap [a_j,b_j]=\emptyset$ for any $i,j \in \{1,\ldots,m\}$ such that $i\ne j$. Then system \eqref{eq:sist-linear} on the half-line is $(\mu,\eps)$-nonuniformly kinematically similar to system $y'=B(t)y$, where $$B(t)=\text{diag}(B_0(t),B_1(t),\cdots,B_{m}(t),B_{m+1}(t)),$$ the matrices
$B_i(t):\,\R_0^+\rightarrow \R^{n_i\times n_i}$, with $n_i=\rank\,\mathcal W_i$, are
differentiable and
\[
\SigmaNDp_\mu(B_i)=
\begin{cases}
\emptyset & \quad \text{for} \quad i\in\{0,m+1\}\\
[a_i,b_i] & \quad \text{for} \quad i\in\{1,\dots,m\}
\end{cases}.
\]
\end{theorem}

In the next example we consider a certain nonautonomous triangular system, compute its nonuniform $\mu_0$-dichotomy spectrum, where $\mu_0(t)=\e^{\sqrt{1+t}-1}$ is a non-exponential growth rate, and obtain its normal form using Theorem~\ref{thm:normal-form-hl}. The normal form, that in this case is diagonal, highlights the contraction/expansion with growth rate $\mu_0$ along linear integral manifolds. We stress that, in the example we will describe, the behavior along integral linear manifolds (and elsewhere) is given by the growth rates $\mu_0(t)$. Thus, this normal form cannot be obtained when we look for exponential behavior (even nonuniform), as confirmed by the fact that, in the present situation, the nonuniform (exponential) dichotomy spectrum doesn't reveal any type of exponential behavior: as we shall see, the nonuniform dichotomy spectrum is $\SigmaNDp(A)=\R$ whereas the nonuniform $\mu_0$-dichotomy spectrum $\SigmaNDp_{\mu_0} = \{-1/2,1/2\}$.

\begin{example}\label{example3}
Consider the bidimensional system
\begin{equation}\label{eq:exemplo-mud-coord}
\left[
\begin{array}{cc}
u'\\
v'
\end{array}
\right]=A(t)
\left[
\begin{array}{cc}
u\\ v
\end{array}
\right]
\quad \text{ with } \quad
A(t)=
\left[
\begin{array}{cc}
-\frac{1}{4\sqrt{1+t}} & \frac{\e^{-\sqrt{1+t}+1}}{2\sqrt{1+t}} \\[2mm]
0 & \frac{1}{4\sqrt{1+t}}
\end{array}
\right].
\end{equation}
Let $\mu_0:\R_0^+\to \R^+$ be the growth rate given by $\mu_0(t)=\e^{\sqrt{1+t}-1}$ and consider in $\R^2$
the norm given by $\|(u,v)\|=|u|+|v|$. The second equation of the system is independent of $u$, which allows us to
obtain explicitly the evolution operator (in matrix form):
\[
\Phi(t,s)=
\left[
\begin{array}{cc}
\left(\frac{\e^{\sqrt{1+s}}}{\e^{\sqrt{1+t}}}\right)^{1/2} & \left(\frac{\e^{\sqrt{1+t}}}{\e^{\sqrt{1+s}}}\right)^{1/2}-\left(\frac{\e^{\sqrt{1+s}}}{\e^{\sqrt{1+t}}}\right)^{1/2}\\[2mm]
0 & \left(\frac{\e^{\sqrt{1+t}}}{\e^{\sqrt{1+s}}}\right)^{1/2}
\end{array}
\right].
\]
For the projections $P_1(t): \mathbb{R}^{2} \rightarrow \mathbb{R}^{2}$ defined by
\[
P_1(t)=
\left[
\begin{array}{cc}
1 & \e^{-\sqrt{1+t}+1}-1\\
0 & 0
\end{array}
\right],
\]
it is easy to check that $\Phi(t,s)P_1(s)=P_1(t)\Phi(t,s)$, for all $s,t \in \R_0^+$.

We have, for $t \ge s \ge 0$,
\[
\begin{split}
\left\|\Phi_\gamma(t,s) P_1(s) (u,v)\right\|
& = \left(\frac{\e^{\sqrt{1+s}}}{\e^{\sqrt{1+t}}}\right)^{1/2-\gamma} \left\|\left(u,(\e^{-\sqrt{1+s}+1}-1)v\right)\right\|\\
& \le \left(\frac{\e^{\sqrt{1+t}}}{\e^{\sqrt{1+s}}}\right)^{\gamma-1/2} \|(u,v)\|
\end{split}
~\]
and, since $Q_1(s)(u,v)=(I-P_1(s))(u,v)=\left((1-\e^{-\sqrt{1+|s|}+1})v,v\right)$, we have, for $0 \le t \le s$,
\[
\begin{split}
\left\|\Phi_\gamma(t,s) Q_1(s) (u,v)\right\|
& \le \left(\frac{\e^{\sqrt{1+t}}}{\e^{\sqrt{1+s}}}\right)^{1/2+\gamma} (2+\e^{1-\sqrt{1+t}})|v|\\
& \le 3\left(\frac{\e^{\sqrt{1+t}}}{\e^{\sqrt{1+s}}}\right)^{1/2+\gamma}\|(u,v)\|.
\end{split}
\]

We conclude that, if $-1/2-\gamma<0$ and $1/2-\gamma>0 $, the system above admits a nonuniform $\mu_0$-dichotomy (that in fact is a uniform $\mu_0$-dichotomy).
Thus, $(-1/2,1/2) \subseteq \rhoND_{\mu_0}(A)$.

For the projections $P_2(t): \mathbb{R}^{2} \rightarrow \mathbb{R}^{2}$ defined by $P_2(t)(u, v)=(u,v)$ we have, for $t \ge s\ge 0$,
\[
\begin{array}{ll}
\left\|\Phi_\gamma(t,s) P_2(s) (u,v)\right\|
& \le \left(\left(\frac{\e^{\sqrt{1+s}}}{\e^{\sqrt{1+t}}}\right)^{\gamma-1/2}+\left(\frac{\e^{\sqrt{1+t}}}{\e^{\sqrt{1+s}}}\right)^{\gamma+1/2}\right)\|(u,v)\|\\[3mm]
& \le 2\left(\frac{\e^{\sqrt{1+t}}}{\e^{\sqrt{1+s}}}\right)^{\gamma-1/2}\|(u,v)\|.
\end{array}
\]
and thus, for $0 \le s \le t$,
\[
\left\|\Phi_\gamma(t,s) P_3(s)\right\|
\le 2\left(\frac{\e^{\sqrt{1+t}}}{\e^{\sqrt{1+s}}}\right)^{\gamma-1/2}.
\]
We conclude that, if $1/2-\gamma<0 \ \Leftrightarrow \ \gamma\in\,(1/2,+\infty)$, we have a nonuniform $\mu_0$-dichotomy.

Finally, for the projections $P_3(t): \mathbb{R}^{2} \rightarrow \mathbb{R}^{2}$ defined by $P_3(t)(u, v)=(0,0)$ we have, for $0 \le t \le s$,
\[
\begin{array}{ll}
\left\|\Phi_\gamma(t,s) Q_3(s) (u,v)\right\|
& \le \left(\frac{\e^{\sqrt{1+t}}}{\e^{\sqrt{1+s}}}\right)^{\gamma-1/2}\|(u,v)\|+\left(\frac{\e^{\sqrt{1+t}}}{\e^{\sqrt{1+s}}}\right)^{1/2+\gamma}\,|v|\\[3mm]
& = 2\left(\frac{\e^{\sqrt{1+t}}}{\e^{\sqrt{1+s}}}\right)^{\gamma+1/2}\|(u,v)\|.
\end{array}
\]
and thus, for $0 \le t \le s$,
\[
\left\|\Phi_\gamma(t,s) Q_3(s)\right\|
\le 2\left(\frac{\e^{\sqrt{1+t}}}{\e^{\sqrt{1+s}}}\right)^{\gamma+1/2}.
\]
We conclude that, if $-1/2-\gamma>0 \ \Leftrightarrow \ \gamma\in\,(-\infty,-1/2)$, we have a nonuniform $\mu_0$-dichotomy.

By the conclusions obtained, we know that $\SigmaNDp_{\mu_0} \subseteq \{-1/2,1/2\}$. On the other hand, the form of $\Phi_{\pm1/2}(t,s)$ for $t,s \in \R_0^+$, shows immediately that for $\gamma=\pm1/2$ we don't have a nonuniform $\mu_0$-dichotomy. Thus
$$\SigmaNDp_{\mu_0} = \{-1/2,1/2\}.$$

In our context, condition~\eqref{eq:3max-min} holds with $\alpha=1/2$ and $\eps=0$ and we have nonuniform bounded growth.
Thus, Theorem~\ref{thm:normal-form-hl} shows that system~\eqref{eq:sist-linear} in the half line is
nonuniformly kinematically similar to a system $y'=B(t)y$, where $B(t)=\text{diag}(B_0(t),B_1(t))$
and
$B_i(t):\,\R_0^+\rightarrow \R$, are differentiable,
$$\SigmaNDp_\mu(B_0)=\{-1/2\} \ \ \text{ and } \ \ \SigmaNDp_\mu(B_1)=\{1/2\}.$$

We can confirm this conclusion by noting that the change of variables
\[
\left[
\begin{array}{c}
u(t) \\[2mm]
v(t)
\end{array}
\right]
=
\left[
\begin{array}{cc}
\sqrt{1+t} & \e^{-\sqrt{1+t}} \\[2mm]
\e^{\sqrt{1+t}-1} & 0
\end{array}
\right]
\left[
\begin{array}{c}
z(t) \\[2mm]
w(t)
\end{array}
\right]
\]
transforms system $x'=A(t)x$ into the diagonal system
\[
\left[
\begin{array}{cc}
z'\\
w'
\end{array}
\right]=B(t)
\left[
\begin{array}{cc}
z\\ w
\end{array}
\right]
\quad \text{ with } \quad
B(t)=
\left[
\begin{array}{cc}
-\frac{1}{4\sqrt{1+t}} & 0\\[2mm]
0 & \frac{1}{4\sqrt{1+t}}
\end{array}
\right],
\]
where it becomes evident that we have polynomial contraction in $\mathcal W_1=\R \times \sg \{(1,0)\}$ and polynomial expansion
in $\mathcal W_2=\R \times \sg \{(0,1)\}$.
\end{example}

\end{document}